%% file: main.tex
\documentclass[11pt, a4paper]{amsart}

\title{A Model Companion for Abelian Lattice-Ordered Groups with a Valuation}
\author{John Stokes-Waters \\ email :- \href{mailto: john.stokes-waters@postgrad.manchester.ac.uk}{john.stokes-waters@postgrad.manchester.ac.uk}}
\date{}

\usepackage{biblatex}
\usepackage[subtle]{savetrees}
\usepackage[left=2cm, right=2cm]{geometry}

\usepackage{imakeidx}
\makeindex

\newcommand{\dwedge}{%
  \mathop{
    \mathchoice{\wedge\mkern-15mu\wedge}
               {\wedge\mkern-12.5mu\wedge}
               {\wedge\mkern-12.5mu\wedge}
               {\wedge\mkern-11mu\wedge}
    }
}

\newcommand{\bdwedge}{%
  \mathop{
    \mathchoice{\bigwedge\mkern-15mu\bigwedge}
               {\bigwedge\mkern-12.5mu\bigwedge}
               {\bigwedge\mkern-12.5mu\bigwedge}
               {\bigwedge\mkern-11mu\bigwedge}
    }
}

\renewcommand{\land}{\hspace{1mm}\dwedge\hspace{1mm}}

\newcommand{\bigland}{\bdwedge}

\usepackage{hyperref}

\usepackage[only,llbracket,rrbracket]{stmaryrd}

\usepackage{amsthm}
\usepackage{amsmath}
\usepackage{amssymb}
\usepackage{mathtools} % :=
\usepackage{faktor} % Quotient notation
\usepackage{marvosym} % Lightning Bolt, for contradiction

\usepackage{bbm}

\usepackage{yfonts}

\usepackage{tikz-cd}

\newtheoremstyle{break}
{}{}%
{\itshape}{}%
{\bfseries}{}%
{\newline}{}

\theoremstyle{break}

\newtheorem{theorem}{Theorem}[section]
\newtheorem{define}[theorem]{Definition}
\newtheorem{lemma}[theorem]{Lemma}
\newtheorem{example}[theorem]{Example}

\newtheorem{remark}[theorem]{Remark}
\newtheorem{corollary}[theorem]{Corollary}

\newtheorem*{theorem*}{Theorem}
\newtheorem*{define*}{Definition}
\newtheorem*{lemma*}{Lemma}
\newtheorem*{example*}{Example}
\newtheorem*{axiom*}{Axiom}
\newtheorem*{remark*}{Remark}
\newtheorem*{corollary*}{Corollary}
\newtheorem*{conjecture*}{Conjecture}

\usepackage{subfiles}

\newcommand{\bb}[1]{\mathord{\mathbb{#1}}}

\newcommand{\bs}[1]{\mathord{\boldsymbol{#1}}}
\newcommand{\cal}[1]{\mathord{\mathcal{#1}}}
\renewcommand{\frak}[1]{\mathord{\mathfrak{#1}}}
\newcommand{\norm}[1]{\mathord{\textnormal{#1}}}

\newcommand{\comment}[1]{}

\renewcommand{\leq}{\leqslant}
\renewcommand{\geq}{\geqslant}
\newcommand{\leqa}{\sqsubseteq}

\usepackage{graphicx}
\makeatletter
\providecommand{\bigsqcap}{%
  \mathop{%
    \mathpalette\@updown\bigsqcup
  }%
}
\newcommand*{\@updown}[2]{%
  \rotatebox[origin=c]{180}{$\m@th#1#2$}%
}
\makeatother

\newcommand{\meet}{\mathop{\text{\raisebox{0.25ex}{\scalebox{0.8}{$\wedge$}}}}}
\newcommand{\meeta}{\mathop{\text{\raisebox{0.25ex}{\scalebox{0.8}{$\sqcap$}}}}}
\newcommand{\bigmeet}{\bigwedge}

\newcommand{\join}{\mathop{\text{\raisebox{0.25ex}{\scalebox{0.8}{$\vee$}}}}}
\newcommand{\joina}{\mathop{\text{\raisebox{0.25ex}{\scalebox{0.8}{$\sqcup$}}}}}
\newcommand{\bigjoin}{\bigvee}

\newcommand{\Cls}[2]{\mathop{\textnormal{Cl}_{#2}(#1)}}
\newcommand{\Cl}[1]{\Cls{#1}{}}

\newcommand{\spectrum}[1]{\textnormal{PrimF}(#1)}

\newcommand{\qcclat}[1]{{\bar{\mathcal{K}}(#1)}} % Closed constructible sets

\newcommand{\patch}[1]{#1_{\textnormal{con}}}

\newcommand{\id}[1]{\mathbbm{1}_{#1}}
\newcommand{\op}[1]{\mathbin{#1^\textnormal{op}}}
\newcommand{\adjoint}{\mathbin{\dashv}}
\DeclareMathOperator{\image}{Im}

\DeclareMathOperator{\Spec}{\underline{\textnormal{Spec}}} % Category of Spectral Spaces

\DeclareMathOperator{\LGrp}{\underline{\ell\textnormal{-Grp}}} % Category of l-Groups
\DeclareMathOperator{\ValGrp}{\underline{\textnormal{Val-}\ell\textnormal{-Grp}}} % Category of valued l-Groups
\DeclareMathOperator{\DValGrp}{\underline{\textnormal{Val-}\ell\textnormal{-Grp}}^{\textnormal{d}}} % Category of valued l-Groups

\newcommand{\abs}[1]{\left| #1 \right|}
\newcommand{\pol}[1]{#1^{\perp}}
\newcommand{\ppol}[1]{#1^{\perp\perp}}
\DeclareMathOperator{\divis}{\textnormal{Divis}}

\newcommand{\quot}[2]{\faktor{#1}{#2}}

\newcommand{\subgrp}{\mathbin{\leq}}
\newcommand{\ideal}{\mathbin{\lhd}}

\newcommand{\minideal}[1]{\mathop{\textnormal{Ideal}\left(#1\right)}}

\newcommand{\lspec}[1]{\ell\textnormal{-Spec}(#1)}
\newcommand{\plspec}[1]{\ell\textnormal{-Spec}^*(#1)}

\newcommand{\valfunc}{\textnormal{Val}}
\newcommand{\val}[1]{\valfunc(\cal{#1})}

\newcommand{\stan}[1]{\textnormal{Stan}\left(#1\right)}

\newcommand{\entails}{\vdash}

\DeclareMathOperator{\lbdlat}{\cal{L}_{\textnormal{BDLat}}}
\DeclareMathOperator{\lbalg}{\cal{L}_{\textnormal{BoolAlg}}}

\DeclareMathOperator{\llgrp}{\cal{L}_{\ell\textnormal{-Grp}}}
\DeclareMathOperator{\lvlgrp}{\frak{L}}
\DeclareMathOperator{\lvlgrpskolem}{\frak{L}^+}

\DeclareMathOperator{\tlgrp}{T_{\ell\textnormal{-Grp}}}
\DeclareMathOperator{\tvlgrp}{\frak{T}}
\DeclareMathOperator{\tdvlgrp}{\frak{T}^{\textnormal{d}}}
\DeclareMathOperator{\tvlgrpac}{\frak{T}^{\textnormal{d}}_{\textnormal{ac}}}
\DeclareMathOperator{\tvlgrpec}{\frak{T}^{\textnormal{d}}_{\textnormal{ec}}}
\DeclareMathOperator{\tvlgrpecskolem}{\frak{T}^{\textnormal{d,+}}_{\textnormal{ec}}}

\newcommand{\typespace}[2]{\textnormal{Stone}_{#2}(#1)}

\begin{document}
    
    \begin{abstract}
        An abelian lattice-ordered group, or abelian \(\ell\)-group, is an abelian group equipped with a compatible lattice ordering. In this paper, we introduce two multi-sorted extensions of abelian lattice-ordered groups inspired by the zero-set maps for continuous functions into \(\bb{R}\). We demonstrate that this expansion is equivalent to equipping \(\cal{G}\) with a spectral subspace \(X\) of \(\lspec{\cal{G}}\), along with the map sending \(a \in \cal{G}\) to \(V(a \meet 0) \cap X\). Using a classical partial quantifier elimination result originally due to Fuxing Shen and Volker Weispfenning \cite{ShenWeispfenning}, we show that one of these expansions admits a model companion, which is complete and has quantifier elimination in a small language expansion.
    \end{abstract}


    \maketitle

    \section{Discussion}

        Abelian \(\ell\)-groups arise frequently throughout mathematics, notably as reducts of rings of real-valued continuous functions on Tychonoff spaces, as well as the value groups of valued fields. In addition, they play an important role in number theory, as \(\bb{Z}^n\) and \(\bb{Q}^n\) are both abelian \(\ell\)-groups for any choice of \(n\). Further, they also play a central role in the study of \L ukasiewicz logic, as the category of abelian \(\ell\)-groups with strong order unit is equivalent to the category of MV algebras. 

        \vspace{1.5mm}

        Our main interests lie in the model-theoretic properties of abelian \(\ell\)-groups. The class of existentially closed abelian \(\ell\)-groups is well-understood (see \cite{SaracinoWood},\cite{AlgClosedAndExtClosedLGrps},\cite{WeispfenningECLGrps}, \cite{ConstructionOfExtClosedLGrpsViaFraisseLimits}), however it does not form an elementary class \cite[Page 72]{WeispfenningECLGrps}. Our goal then is to find a suitable theory of abelian \(\ell\)-groups which admits a model companion. 

        \vspace{1.5mm}
        
        In this paper, we propose a multi-sorted extension to abelian \(\ell\)-groups, which we call \textit{densely valued} \(\ell\)\textit{-groups}. These are inspired by the ring of real-valued continuous functions \(C(\bb{R})\), and the map \(P\) which assigns to each \(f \in C(\bb{R})\) the set \(\{x \in \bb{R} \ | \ f(x) \geq 0\}\). As such, a densely valued \(\ell\)-group is a triple \((\cal{G},\cal{L},P)\) where \(\cal{G}\) is an abelian \(\ell\)-group, \(\cal{L}\) is a bounded distributive lattice, and \(P\) is a lattice morphism from \(\cal{G}\) to \(\cal{L}\) satisfying some additional axioms (See Definition \ref{ValuationDef} for details). Further, we show that there is a pair of adjoint functors between the category of abelian \(\ell\)-groups and the category of densely valued \(\ell\)-groups (Theorem \ref{NatValAdjoint}). This adjunction equips each abelian \(\ell\)-group \(\cal{G}\) with the lattice of closed constructible sets of its spectrum. 

        \vspace{1.5mm}

        The main importance of densely valued \(\ell\)-groups is that they are the natural model-theoretic environment for the application of the Shen-Weispfenning theorem - a classical result due to \cite{ShenWeispfenning} and stated in Theorem \ref{ShenWeispfenning}, which tells us that if \((\cal{G},\cal{L},P)\) is densely valued with \(\cal{G}\) divisible and satisfying \textit{patching} (see Definition \ref{PatchingDef}), then we can eliminate quantifiers in the group sort. 

        \vspace{1.5mm}

        Using the Shen-Weispfenning theorem, we first characterise the algebraically closed densely valued \(\ell\)-groups (Theorem \ref{ACValGrpAltDef}). In particular, they are those triples \((\cal{G},\cal{L},P)\) with patching, \(\cal{G}\) divisible, and \(\cal{L}\) Boolean. Then, we use this to characterise existentially closed densely valued \(\ell\)-groups (Theorem \ref{ECValGrpAltDef}) as those which are algebraically closed and with \(\cal{L}\) atomless. 

        \vspace{1.5mm}

        We investigate the theory \(\tvlgrpec\) of densely valued \(\ell\)-groups, and show it is complete (Theorem \ref{ECDVAlGrpIsComp}), and has quantifier elimination in the language \(((+,-,\leq,0,\meet,\join),(\leqa,\joina,\meeta,\top,\bot,\lnot),P)\) (Theorem \ref{QEForECValGrpsWBoolComp}).

    \section{Preliminaries}

        In this section, we review some essential notions from the theory of lattice-ordered groups. For more information, see \cite{LGroupTextbook1}, \cite{LGroupTextbook2}. 

        \vspace{3mm}

        An \textbf{abelian} \(\bs{\ell}\)\textbf{-group} is an abelian group \(\cal{G}\) with a partial order \(\leq\) such that \((\cal{G},\leq)\) is a lattice, and such that, for \(a,b,c \in \cal{G}\) with \(a \leq b\), we have that \(a + c \leq b + c\). From now on, we use ``\(\ell\)-group" to mean ``abelian \(\ell\)-group". If \(\leq\) is a total order, then we say that \(\cal{G}\) is an \(\bs{o}\)\textbf{-group}. We note that the class of \(\ell\)-groups is finitely axiomatisable in the language \(\llgrp = (+,-,0,\leq,\meet,\join)\), and we denote the corresponding theory by \(\tlgrp\). 

        \vspace{1.5mm}
        
        For an element \(a \in \cal{G}\), we write \(\abs{a} = a \join (-a)\), which is the \textbf{absolute value} of \(a\). We note that \(0 \leq \abs{a}\) for all \(a \in \cal{G}\). Further, we write \(\cal{G}^+ = \{a \in \cal{G} \ | \ 0 \leq a\}\) for the \textbf{positive cone} of \(\cal{G}\). 

        \vspace{1.5mm}
        
        An \(\bs{\ell}\)\textbf{-group morphism} is a map \(\phi : \cal{G} \to \cal{H}\) which is both a group morphism and a lattice morphism. We say \(\phi\) is an \textbf{embedding} (resp. \textbf{surjection}) if it is injective (resp. surjective). 

        \vspace{1.5mm}
        
        An \(\ell\)-group \(\cal{G}\) is \textbf{divisible} if, for all \(a \in \cal{G}\) and \(n \in \bb{N}\), there exists \(b \in \cal{G}\) with \(nb = a\). If \(\cal{G}\) is a divisible \(o\)-group, then we call it a \textbf{DOAG}. The \textbf{divisible hull} of \(\cal{G}\), denoted \(\divis(\cal{G})\), is the tensor product \(\bb{Q} \otimes \cal{G}\). We formally identify \(\divis(\cal{G})\) with the set of symbols \(\left\{\frac{a}n \ | \ a \in \cal{G}, \ n \in \bb{N}\right\}\), with operations given for \(a,b \in \cal{G}\) and \(n,m \in \bb{N}\) by \(\frac{a}n - \frac{b}m =\frac{a - b}{nm}\), \(\frac{a}n \join \frac{b}m = \frac{a \join b}{nm}\), and \(\frac{a}n \meet \frac{b}m = \frac{a \meet b}{nm}\). The map \(\iota : \cal{G} \to \divis(\cal{G}) : g \mapsto \frac{g}1\) is an \(\ell\)-group embedding, and \(\divis(\cal{G})\) embeds into any divisible extension of \(\cal{G}\).

        \vspace{1.5mm}
        
        An \(\bs{\ell}\)\textbf{-subgroup} of \(\cal{G}\) is a subgroup \(\cal{H} \subseteq \cal{G}\) that is also a sublattice. We denote this by \(\cal{H} \leq \cal{G}\). An \(\ell\)-subgroup \(\cal{J} \subgrp \cal{G}\) is an \(\bs{\ell}\)\textbf{-ideal} if it is also \textbf{convex} - for all \(a,b \in \cal{J}\) and \(c \in \cal{G}\) with \(a \leq c \leq b\), \(c \in \cal{J}\). We denote this by \(\cal{J} \ideal \cal{G}\), and note that \(\ell\)-ideals are precisely the kernels of \(\ell\)-group morphisms. For \(S \subseteq \cal{G}\), the \textbf{minimal} \(\bs{\ell}\)\textbf{-ideal} containing \(S\), denoted \(\minideal{S}\), is the intersection of all \(\ell\)-ideals containing \(S\).

        \vspace{1.5mm}
        
        An \(\ell\)-ideal \(\cal{J} \ideal \cal{G}\) is \textbf{prime} if any of the following equivalent conditions hold:\((1) \ \quot{\cal{G}}{\cal{J}}\) is totally-ordered; \((2)\) For all \(a,b \in \cal{G}^+\) with \(a \meet b \in \cal{J}\), \(a \in \cal{J}\) or \(b \in \cal{J}\); or \((3)\) For all \(a,b \in \cal{G}\) with \(a \meet b = 0\), \(a \in \cal{J}\) or \(b \in \cal{J}\).

        \vspace{1.5mm}

        We remark that we adopt the convention that \(\cal{G}\) is a prime \(\ell\)-ideal of \(\cal{G}\). We write \(\lspec{\cal{G}}\) for the collection of prime \(\ell\)-ideals of \(\cal{G}\). This is a spectral space with quasicompact open sets \(D(a) = \{\cal{J} \in \lspec{\cal{G}} \ | \ a \notin \cal{J}\} \cup \{\lspec{\cal{G}}\}\). Further, we let \(V(a) = \lspec{\cal{G}} \setminus D(a)\). We define both \(D(S)\) and \(V(S)\) for arbitrary \(S \subseteq \cal{G}\) similarly, and note that \(\minideal{S} = \bigcap V(S)\). This is the \textbf{Nullstellensatz} (for \(\ell\)-groups).

        \vspace{1.5mm}

        Let \(\LGrp\) be the category of \(\ell\)-groups, and \(\Spec\) the category of spectral spaces. Then, the map \(\ell\norm{-Spec} : \LGrp \to \Spec\) given for an \(\ell\)-group morphism \(\phi : \cal{G} \to \cal{H}\) and \(\cal{J} \in \lspec{\cal{H}}\) by \(\lspec{\phi}(\cal{J}) = \phi^{-1}(\cal{J})\) is a contravariant functor. Further, \(\ell\)-Spec sends embeddings to surjections, surjections to embeddings, and for \(a \in \cal{G}\), we have \(\lspec{\phi}^{-1}(V(a)) = V(\phi(a))\).

        \vspace{1.5mm}

        For an \(\ell\)-group \(\cal{G}\) and any \(X \subseteq \lspec{\cal{G}}\) with \(\bigcap X = \{0\}\), we have an \(\ell\)-group embedding:

        \[I : \cal{G} \to \prod_{\cal{J} \in X} \quot{\cal{G}}{\cal{J}} : a \mapsto (a + \cal{J})_{\cal{J} \in X}.\]

        Further, by taking divisible hulls in each component, we can embed \(\cal{G}\) in a product of DOAGs. The theory of DOAGs is complete, and so by taking a saturated extension, we get that there is some DOAG \(\Gamma\) and \(\ell\)-group embedding \(\cal{G} \to \Gamma^X\).

        \vspace{1.5mm}

        Given a set \(S \subseteq \cal{G}\), the \textbf{polar} of \(S\) is the set \(\pol{S} = \{g \in \cal{G} \ | \ (\forall s \in S)(g \meet s = 0)\}\). For an element \(g \in \cal{G}\), we write \(\pol{g} = \pol{\{g\}}\), and \(\ppol{g} = \pol{(\pol{g})}\). 

        \vspace{1.5mm}

        Originally due to \cite[Page \(69\)]{WeispfenningECLGrps}, with the below modern treatment due to \cite[Section 2.1]{ConstructionOfExtClosedLGrpsViaFraisseLimits}, an \(\ell\)-group \(\cal{G}\) is existentially closed if it satisfies the following:

        \begin{itemize}
            \item[\((1)\)] \(\cal{G}\) is divisible.

            \item[\((2)\)] For all \(a,b,c \in \cal{G}^+\) with \(a \meet b = 0\), there exists \(g,h \in \cal{G}\) with \(c = g + h\), and:
            
            \[a \meet h = g \meet b = g \meet h = 0.\]

            \item[\((3)\)] For all \(a \in \cal{G}^+\) non-zero, there exists \(g,h \in \cal{G}^+\) non-zero such that \(g,h \leq a\) and \(g \meet h = 0\).

            \item[\((4)\)] For all \(a \in \cal{G}^+\), there exists \(g \in \cal{G}^+\) non-zero such that \(a \meet g = 0\).

            \item[\((5)\)] \(\minideal{g} = \ppol{g}\) for all \(g \in \cal{G}\).
        \end{itemize}

        If only \((1)\) and \((2)\) hold, then \(\cal{G}\) is algebraically closed. We remark that conditions \((1) - (4)\) are first-order in \(\llgrp\), and so \((5)\) isn't by \cite[Page \(72\)]{WeispfenningECLGrps} and general model theory.

        \vspace{1.5mm}

        An element \(g \in \cal{G}^+\) is a \textbf{weak order unit} if either of the following equivalent conditions holds: \((1)\) For all \(a \in \cal{G}^+\), if \(g \meet a = 0\), then \(a = 0\); or \((2)\) The set \(D(g)\) is dense in \(\lspec{\cal{G}}\).

        \vspace{1.5mm}

        Further, it is a \textbf{strong order unit} if either of the following equivalent conditions holds: \((1)\) For all \(a \in \cal{G}^+\), there exists some \(n \in \bb{N}\) such that \(a \leq ng\); or \((2)\) \(D(g) = \lspec{\cal{G}} \setminus \{\cal{G}\}\).

        \vspace{1.5mm}

        We call \(\cal{G}\) is \textbf{Archimedean} if every \(g \in \cal{G}^+\) is a strong order unit, and it is \textbf{hyper-Archimedean} if, for all surjective \(\ell\)-group morphism \(\phi : \cal{G} \to \cal{H}\), \(\cal{H}\) is Archimedean. Equivalently, \(\cal{G}\) is hyper-Archimedean if every proper \(\ell\)-ideal is minimal. 

        \vspace{3mm}

        In addition to the above, we will make use of results in lattice theory \cite{LatTheoryTextbook}, model theory \cite{Hodges}\cite{ChangAndKeisler}, theory of spectral spaces \cite{TheBible}, and category theory \cite{MacLane}. We remark that, while both \cite{Hodges} and \cite{ChangAndKeisler} provide only a single-sorted treatement of model theory, every result we utilise carries over into the multi-sorted setting. Further, for clarity, we fix some notation:
        
        \begin{itemize}
            \item For \(\cal{L}\) a bounded distributive lattice, and \(X\) a spectral space, we write \(\mu_{\cal{L}} : \cal{L} \to \qcclat{\spectrum{\cal{L}}} : l \mapsto V(l)\) and \(\nu_X : X \to \spectrum{\qcclat{X}} : x \mapsto \{C \in \qcclat{X} \ | \ x \in C\}\) for the natural isomorphisms of spectral duality. These will only be used for technical reasons, in order to ensure maps are explicit and well-defined.

            \item \(\lbdlat = (\leq,\meet,\join,\bot,\top)\) is the single-sorted language of bounded distributive lattices. 

            \item For \(X = (X,\tau)\) a topological space, and \(S \subseteq X\), \(\Cls{S}{\tau}\) is the closure of \(S\) in \(X\). Where clear, we write \(\Cl{S}\) for \(\Cls{S}{\tau}\).

            \item For \(X = (X,\tau)\) a spectral space, the \textbf{patch space} of \(X\) is denoted \(\patch{X}\), with the corresponding topology \(\patch{\tau}\) generated by the basis \(\{U \cap C \ | \ U, X \setminus C \subseteq X \text{ quasicompact open }\}\). 
        \end{itemize}

    \section{Valuations on \(\ell\)-Groups}

\subfile{Chapters/valued-l-groups}

    \section{Conrad-Darnel-Nelson Valuations}

\subfile{Chapters/existing-literature}

    \section{Standard Structures}

\subfile{Chapters/standard-structures}

    \section{The Model Theory of Existentially Closed Densely Valued \(\ell\)-Groups}

\subfile{Chapters/model-theory-of-ec-valued-l-groups}

    \printbibliography
        
\end{document}

%% file: Chapters/valued-l-groups.tex
    \subsection{Motivation}

        As discussed, the prime ideal structure of \(\cal{G}\) is of pivotal importance in understanding \(\cal{G}\) as a whole. Therefore, we want a valuation on an \(\ell\)-group \(\cal{G}\) to be a tool we can use to examine this structure in a first-order setting. The following example gives some motivation as to how we will go about doing this.
            
        \begin{example}
            Let \(C(\bb{R})\) be the \(\ell\)-group of continuous functions from \(\bb{R}\) to \(\bb{R}\), where all operations are carried out pointwise. The set \(\cal{J}_x = \{f \in C(\bb{R}) \ | \ f(x) = 0\}\) is clearly a prime \(\ell\)-ideal of \(C(\bb{R})\). 
        \end{example}

        Hence, in order to gain some access to prime \(\ell\)-ideals in a first-order setting, we would like to construct a notion of valuation that resembles the map \(Z\) sending a continuous function \(f : \bb{R} \to \bb{R}\) to the set \(f^{-1}(0) \subseteq \bb{R}\). However, we observe that this map doesn't respect of any \(\cal{G}\)'s structure. Indeed, we see:
        \[Z(X) \cap Z(X + 1) = \emptyset \neq \{0\} = Z(X) = Z(X \meet (X + 1)).\]

        If however, we instead consider the map \(P\) which sends a continuous function \(f : \bb{R} \to \bb{R}\) to the set \(f^{-1}([0,\infty)) \subseteq \bb{R}\) of points where \(f\) takes on a positive value, then we observe the following:
        
        \begin{itemize}
            \item[\((1)\)] For any \(f,g : \bb{R} \to \bb{R}\) continuous, we have that:
            \begin{align*}
                0 \leq f(x) \text{ and } 0 \leq g(x) \iff & 0 \leq \min \{f(x),g(x)\} \\
                0 \leq f(x) \text{ or } 0 \leq g(x) \iff & 0 \leq \max \{f(x),g(x)\}.
            \end{align*}

            In particular, the map \(P\) preserves the lattice structure on \(\cal{G}\).

            \item[\((2)\)] We can recover \(Z\) in a first-order way by noticing that:
            \[f(x) = 0 \iff 0 \leq (-\abs{f})(x).\]
        \end{itemize}

        In general, an \(\ell\)-group isn't an \(\ell\)-group of continuous functions. Instead, we use the following folklore result to provide a general framework for constructing a valuation.
        
        \begin{theorem} \label{NatFuncEmbedOfLGrps}
            Let \(\cal{G}\) be an \(\ell\)-group. Then, there are \(\ell\)-group embeddings:
            \[\cal{G} \hookrightarrow \prod_{\cal{J} \in \lspec{\cal{G}}} \quot{\cal{G}}{\cal{J}} \hookrightarrow \Gamma^{\lspec{\cal{G}}}\]

            for some divisible \(o\)-group \(\Gamma\). We remark that this map isn't uniquely determined, but for any choice of composition \(I\), we have that:
            \[I(f)(\cal{J}) \geq 0 \iff I(f \meet 0)(\cal{J}) = 0 
            \iff f \meet 0 \in \cal{J}
            \iff \cal{J} \in V(f \meet 0)\]

            for any \(f \in \cal{G}\) and \(\cal{J} \in \lspec{\cal{G}}\).
        \end{theorem}

        This justifies asking the question ``At what points of its domain is an element of an \(\ell\)-group positive?" in general. Thus, we construct our definition of a valuation on an \(\ell\)-group:
        
        \begin{define} [Valuation on an \(\ell\)-Group]\label{ValuationDef}
            Let \(\cal{G}\) be an \(\ell\)-group. Then, a \textbf{valuation} on \(\cal{G}\) is a lattice morphism:
            \[P : \cal{G} \to \cal{L}\]

            where \(\cal{L}\) is a bounded distributive lattice, satisfying the following:
            
            \begin{itemize}
                \item \(P\) is \textbf{scaling invariant} - \(P(a) = P(na)\) for all \(a \in \cal{G}\) and \(n \in \bb{N}\).

                \item \(P\) is \textbf{positivity affirming} - \(P(a) = \top\) for all \(a \in \cal{G}^+\).

                \item \(P\) is \textbf{essentially total} - for all \(l \in \cal{L}\) with \(l \neq \bot\), there exists some \(a \in \cal{G}\) with \(P(a) = l\).
            \end{itemize}

            Further, if for all \(a \in \cal{G}\), we have that:
            \[P(a) = \top \implies 0 \leq a\]

            then we say that \(P\) is \textbf{positivity detecting}, and a \textbf{dense valuation}.
        \end{define}

        \begin{example}
            Let \(\cal{G}\) be an \(\ell\)-group, and \(\cal{L} = \cal{P}(\lspec{\cal{G}})\). Then, the map:
            \[P : \cal{G} \to \cal{L} : a \mapsto \{\cal{J} \in \lspec{\cal{G}} \ | \ a \meet 0 \in \cal{J}\}\]

            is a dense valuation on \(\cal{G}\). Equivalently, by Theorem \ref{NatFuncEmbedOfLGrps}, we have that:
            \[P(a) = \{\cal{J} \in \lspec{\cal{G}} \ | \ I(a)(\cal{J}) \geq 0\}.\]

            We remark that \(P\) is not surjective - indeed, for every \(a \in \cal{G}\), there is some \(\cal{J} \in \lspec{\cal{G}}\) with \(a \meet 0 \in \cal{J}\), and so \(P(a) \neq \emptyset\).
        \end{example}

        We can bundle an \(\ell\)-group with a valuation to form a multi-sorted structure \((\cal{G},\cal{L},P)\), which we call a \textbf{valued} \(\bs{\ell}\)\textbf{-group}, or a \textbf{densely valued} \(\bs{\ell}\)\textbf{-group} if \(P\) is a dense valuation. In particular, we view this in the language \(\lvlgrp\) consisting of two sorts, \(\bb{G}\) and \(\bb{L}\), and symbols:
        
        \begin{center}
            \begin{tabular}{c|c|c}
                Constant Symbols & Relation Symbols & Function Symbols \\
                \hline
                \(0 : \bb{G}\) & \(\leq : \bb{G} \times \bb{G}\) & \(- : \bb{G} \to \bb{G}\) \\
                \(\top : \bb{L}\) & \(\leqa : \bb{L} \times \bb{L}\) & \(+ : \bb{G} \times \bb{G} \to \bb{G}\) \\
                \(\bot : \bb{L}\) & & \(\meet : \bb{G} \times \bb{G} \to \bb{G}\) \\
                & & \(\join : \bb{G} \times \bb{G} \to \bb{G}\) \\
                & & \(\meeta : \bb{L} \times \bb{L} \to \bb{L}\) \\
                & & \(\joina : \bb{L} \times \bb{L} \to \bb{L}\) \\
                & & \(P : \bb{G} \to \bb{L}\)
            \end{tabular}
        \end{center}

        In this language, there is evidently a theory of both valued and densely \(\ell\)-groups, which we denote \(\tvlgrp\) and \(\tdvlgrp\) respectively. We observe that these theories are \(\forall\exists\). This also gives us the notion of a valued \(\ell\)-group morphism for free.

        \begin{define} [Valued \(\ell\)-Group Morphism]
            Let \(\cal{V} = (\cal{G},\cal{L},P)\) and \(\cal{W} = (\cal{H},\cal{K},Q)\) be valued \(\ell\)-groups. Then, a \textbf{valued} \(\bs{\ell}\)\textbf{-group morphism} from \(\cal{V}\) to \(\cal{W}\) is a pair \((\phi,\psi)\) consisting of an \(\ell\)-group morphism and a bounded lattice morphism, and such that \(Q \circ \phi = \psi \circ P\).

            Further, \((\phi,\psi)\) is an \textbf{embedding} (resp. \textbf{surjection}) if both \(\phi\) and \(\psi\) are embeddings (resp. surjections).
        \end{define}

        We write \(\ValGrp\) for the category of valued \(\ell\)-groups and valued \(\ell\)-group morphisms, and \(\DValGrp\) for the full subcategory whose objects are densely valued \(\ell\)-groups.

    \subsection{The Natural Valuation of an \(\ell\)-Group}

        Theorem \ref{NatFuncEmbedOfLGrps} suggests that there is a natural way to associate a valuation to an \(\ell\)-group, and that this can be done in such a way that the associated valuation looks precisely like taking a function to the set of points where it has a positive value. It turns out that not only is this true, but this construction is functorial. 

        \begin{lemma} \label{NatValOnFuncs}
            Let \(\cal{G}\) be an \(\ell\)-group. Then:
            \begin{itemize}
                \item[\((1)\)] The triple \(\val{G} \coloneqq (\cal{G},\qcclat{\lspec{\cal{G}}},V(- \meet 0))\) is a densely valued \(\ell\)-group.

                \item[\((2)\)] If \(\cal{H}\) is an \(\ell\)-group, and \(\phi : \cal{G} \to \cal{H}\) an \(\ell\)-group morphism, then there is a valued \(\ell\)-group morphism:
                \[\val{\phi} \coloneqq (\phi,\qcclat{\lspec{\phi}}) : \val{G} \to \val{H}.\]

                \item[\((3)\)] If \(\phi\) is an embedding (resp. surjection), then so is \(\val{\phi}\).

                \item[\((4)\)] \(\valfunc\) is a covariant functor from \(\LGrp\) to \(\DValGrp\).
            \end{itemize}
        \end{lemma}
        \begin{proof}
            \((1)\) Obvious from the observation that \(\cal{J} \in V(f \meet 0)\) if and only if \(I(f)(\cal{J}) \geq 0\), where \(I\) is an \(\ell\)-group embedding of \(\cal{G}\) in \(\Gamma^{\lspec{\cal{G}}}\) as in Theorem \ref{NatFuncEmbedOfLGrps}. \\
        
            \((2)\) It is clear that \(\qcclat{\lspec{\phi}}\) is a bounded lattice morphism. Hence, it suffices to show that this diagram commutes:
            \begin{center}
                \begin{tikzcd}
                    \cal{G} \arrow[rrr, "\phi" description] \arrow[dd, "V(- \meet 0)" description] &  &  & \cal{H} \arrow[dd, "V(- \meet 0)" description] \\
                    &  &  &                                                \\
                    \qcclat{\lspec{\cal{G}}} \arrow[rrr, "\qcclat{\lspec{\phi}}" description]      &  &  & \qcclat{\lspec{\cal{H}}}                      
                \end{tikzcd}
            \end{center}

            Thus, let \(a \in \cal{G}\) and \(\cal{J} \in \lspec{\cal{H}}\). We see that:
            \begin{align*}
                \cal{J} \in \qcclat{\lspec{\phi}}(V(a \meet 0)) \iff & \lspec{\phi}(\cal{J}) \in V(a \meet 0) \\
                \iff & a \meet 0 \in \lspec{\phi}(\cal{J}) \\
                \iff & \phi(a \meet 0) \in \cal{J} \\
                \iff & \cal{J} \in V(\phi(a) \meet 0)
            \end{align*}

            as required. \\
            
            \((3)\) Immediate, as the functors \(\bar{\cal{K}}\) and \(\ell\)-Spec both have this property. \\

            \((4)\) As \(\bar{\cal{K}}\) and \(\ell\)-Spec are both contravariant functors, then their composition is a covariant functor, and so we are done.
        \end{proof}

        Even stronger than being a functor, we can show that \(\valfunc\) is a left adjoint to the forgetful functor \(\iota : \DValGrp \to \LGrp\). To do this, we first prove the following lemma.

        \begin{lemma} \label{NatValFuncExtensions}
            Let \(\cal{V} \coloneqq (\cal{G},\cal{L},P)\) be a valued \(\ell\)-group. Then, there is a unique bounded lattice morphism \(\epsilon_{\cal{V}} : \qcclat{\lspec{\cal{G}}} \to \cal{L}\) such that this diagram commutes:
            \begin{center}
                \begin{tikzcd}
                    \cal{G} \arrow[rrdd, "P" description] \arrow[dd, "V(- \meet 0)" description] &  &         \\
                    &  &         \\
                    \qcclat{\lspec{\cal{G}}} \arrow[rr, "\epsilon_{\cal{V}}" description]                      &  & \cal{L}
                \end{tikzcd}
            \end{center}

            Further, \(\epsilon_{\cal{V}}\) is surjective and given by \(\epsilon_{\cal{V}}(V(a)) = P(-\abs{a})\) and \(\epsilon_{\cal{V}}(\emptyset) = \bot\).
        \end{lemma}
        \begin{proof}
            We first show that \(\epsilon_{\cal{V}}\) is well-defined. Hence, let \(a,b \in \cal{G}\) such that \(V(a) = V(b)\). By the Nullstellensatz, we have that \(\minideal{a} = \minideal{b}\), and so there exists \(n,m \in \bb{N}\) such that \(\abs{a} \leq n\abs{b}\) and \(\abs{b} \leq m\abs{a}\). In particular, \(-n\abs{b} \leq -\abs{a}\) and \(-m\abs{a} \leq -\abs{b}\), and so:
            \[P(-\abs{b}) = P(-n\abs{b}) \leqa P(-\abs{a}) = P(-m\abs{a}) \leqa P(-\abs{a}).\]
            
            Thus, \(\epsilon_{\cal{V}}(V(a)) = \epsilon_{\cal{V}}(V(b))\). Next, we show that \(\epsilon_{\cal{V}}\) is a bounded lattice morphism. Hence, let \(a,b \in \cal{G}^+\) and consider that:
            \begin{align*}
                \epsilon_{\cal{V}}(V(a) \cup V(b)) = & \epsilon_{\cal{V}}(V(a \meet b)) \\
                = & P(-(a \meet b)) \\
                = & P(-a) \joina P(-b) \\
                = & \epsilon_{\cal{V}}(V(a)) \joina \epsilon_{\cal{V}}(V(b)).
            \end{align*}

            Hence, \(\epsilon_{\cal{V}}\) respects joins, and mutatis mutandis respects meets. By definition, \(\epsilon_{\cal{V}}(\emptyset) = \bot\), , and we note that \(\epsilon_{\cal{V}}(V(0)) = P(-\abs{0}) = \top\).  Thus, \(\epsilon_{\cal{V}}\) is a bounded lattice morphism. Next, we show that it uniquely makes this diagram commute:
            \begin{center}
                \begin{tikzcd}
                    & \cal{G} \arrow[rd, "P" description] \arrow[ld, "V(- \meet 0)" description] &         \\
                    \qcclat{\lspec{\cal{G}}} \arrow[rr, "\epsilon_{\cal{V}}" description] &                                                                            & \cal{L}
                \end{tikzcd}
            \end{center}

            For \(a \in \cal{G}\), we see that:
            \[\epsilon_{\cal{V}}(V(a \meet 0)) = P(-\abs{a \meet 0}) = P(-((-a) \join 0)) = P(a \meet 0) = P(a).\]

            Hence, the above diagram commutes. Further, if \(\gamma : \qcclat{\lspec{\cal{G}}} \to \cal{L}\) also makes it commute, then we see that for \(a \in \cal{G}\):
            \[\gamma(V(a)) = \gamma(V(-\abs{a})) = \gamma(V((-\abs{a}) \meet 0)) = P(-\abs{a}) = \epsilon_{\cal{V}}(V(a)).\]

            Hence, \(\gamma = \epsilon_{\cal{V}}\), and so \(\epsilon_{\cal{V}}\) is unique. Finally, let \(l \in \cal{L}\). If \(l = \bot\), then \(\epsilon_{\cal{V}}(\emptyset) = l\). Otherwise, as \(P\) is essentially total, then there exists \(a \in \cal{G}\) such that \(P(a) = l\). By the above, we have that \(l = P(a) = \epsilon_{\cal{V}}(V(a \meet 0))\), as required.
        \end{proof}

        Using this, it is almost immediate that \(\valfunc\) is left adjoint to \(\iota\). Regardless, we include a proof for completeness.

        \begin{corollary} [The Natural Valuation Adjunction]\label{NatValAdjoint}
            The functor \(\valfunc\) is left adjoint to forgetful functors \(\iota^{\norm{d}} : \DValGrp \to \LGrp\) and \(\iota : \ValGrp \to \LGrp\).
        \end{corollary}
        \begin{proof}
            We show the case for \(\iota\), with the case for \(\iota^{\norm{d}}\) following similarly. Let \(\cal{V} \coloneqq (\cal{G},\cal{L},P)\) be a valued \(\ell\)-group, \(\cal{H}\) an \(\ell\)-group, and \((\phi,\psi) : \val{H} \to \cal{V}\) a valued \(\ell\)-group morphism. We claim that \(\phi\) is the unique \(\ell\)-group morphism such that the diagram \((\dagger)\) below commutes:
            \begin{center}
                \begin{tikzcd}
                    \val{H} \arrow[rrdd, "{(\phi,\psi)}" description] \arrow[dd, "\valfunc(\phi)" description] &  &         \\
                    &  &         \\
                    \val{G} \arrow[rr, "{(\id{\cal{G}},\epsilon_{\cal{V}})}" description]                      &  & \cal{V}
                \end{tikzcd}
            \end{center}

            where \(\epsilon_{\cal{V}}\) is the map per Lemma \ref{NatValFuncExtensions}. To show commutativity of \((\dagger)\), it is sufficient to show that these two diagrams commute:
            \begin{center}
                \begin{tikzcd}
                    \cal{H} \arrow[rrdd, "\phi" description] \arrow[dd, "\phi" description] &  &         &  & \qcclat{\lspec{\cal{H}}} \arrow[rrdd, "\psi" description] \arrow[dd, "\qcclat{\lspec{\phi}}" description] &  &         \\
                    &  &         &  &                                                                                                           &  &         \\
                    \cal{G} \arrow[rr, "\id{\cal{G}}" description]                          &  & \cal{G} &  & \qcclat{\lspec{\cal{G}}} \arrow[rr, "\epsilon_{\cal{V}}" description]                                     &  & \cal{L}
                \end{tikzcd}
            \end{center} 

            Clearly, the leftmost diagram commutes. Hence, let \(a \in \cal{H}\) and consider:
            \begin{align*}
                \psi(V(a)) = & \psi(V((-\abs{a}) \meet 0)) \\
                = & (P \circ \phi)(-\abs{a}) \\
                = & P(-\abs{\phi(a)}) \\
                = & \epsilon_{\cal{V}}(V(\phi(a)) \\
                = & (\epsilon_{\cal{V}} \circ \qcclat{\lspec{\phi}})(V(a)).
            \end{align*}

            Thus, \((\dagger)\) commutes. Finally, if the map \(\alpha: \cal{H} \to \cal{G}\) is another \(\ell\)-group morphism such that \(\valfunc(\alpha)\) makes \((\dagger)\) commute, then we remark that it is immediate that, for \(a \in \cal{H}\):
            \[\phi(a) = (\id{\cal{G}} \circ \alpha)(a) = \alpha(a).\]

            Hence, \(\phi = \alpha\), and so by definition, \(\valfunc \adjoint \iota\). 
        \end{proof}

        This adjunction mimics the case of valued fields - in particular, it is clear that the map which equips a field with the trivial valuation \(\nu : K \to \bb{Z} : k \mapsto 0\) is left adjoint to the forgetful functor. However, we notice that here the natural valuation of an \(\ell\)-group is much more interesting - the lattice \(\qcclat{\lspec{\cal{G}}}\) carries non-trivial information about \(\cal{G}\) that can't be accessed in a first-order way in \(\cal{G}\) alone. 

    \subsection{Topological Representations of Valued \(\ell\)-Groups}

        The previous section suggests that valuations on \(\ell\)-groups are connected to their \(\ell\)-spectra. Indeed, we can construct a valuation on an \(\ell\)-group from any spectral subspace of its \(\ell\)-spectrum, as the next lemma demonstrates.

        \begin{define} [Dense Set]
            Let \(X = (X,\tau)\) be a topological space, and \(D \subseteq X\). Then, \(D\) is \textbf{dense in } \(\bs{X}\) if \(\Cls{D}{\tau} = X\).
        \end{define}

        \begin{lemma} \label{ValGrpsFromDenseSpecSubspaces}
            Let \(\cal{G}\) be an \(\ell\)-group, and \(X \subseteq \lspec{\cal{G}}\) be a spectral subspace. Then, we have a valuation \(P : \cal{G} \to \qcclat{X}\) given by \(P(a) = V(a \meet 0) \cap X\). Further, if \(X\) is dense in \(\lspec{\cal{G}}\), then \(P\) is a dense valuation.
        \end{lemma}
        \begin{proof}
            As \(V(- \meet 0) : \cal{G} \to \qcclat{\lspec{\cal{G}}}\) is a valuation, it is clear that \(P\) is a scaling invariant and essentially total lattice morphism. Further, we notice that for \(a \in \cal{G}^+\), if \(V(a \meet 0) = \lspec{\cal{G}}\), then \(P(a) = X\). Thus, \(P\) is a valuation. 
            
            Now, if \(X\) is dense, then \(X \cap D(a) \neq \emptyset\) for any non-zero \(a \in \cal{G}\). Thus, if \(P(a) = x\), then \(X \subseteq V(a \meet 0)\). Hence, \(X \cap D(a \meet 0) = \emptyset\), and so by density of \(X\), \(a \meet 0 = 0\). Thus, \(0 \leq a\), and so \(P\) is a dense valuation.
        \end{proof}

        This is already a powerful tool for producing examples of valued \(\ell\)-groups. However, the remarkable thing is that every valuation on \(\cal{G}\) is of this form.

        \begin{lemma} \label{DenseSubspaceAltDef}
            Let \(\cal{G}\) be an \(\ell\)-group, and \(X \subseteq \lspec{\cal{G}}\) a spectral subspace. Then, \(X\) is dense in \(\lspec{\cal{G}}\) if and only if \(\bigcap X = \{0\}\).
        \end{lemma}
        \begin{proof}
            \((\implies)\) Let \(a \in \bigcap X\). By definition, we have for all \(\cal{J} \in X\) that \(a \in \cal{J}\). Thus, \(X \subseteq V(a)\), and so we see that \(X = V(a) \cap X = V((-\abs{a}) \meet 0) \cap X\). Therefore, \(0 \leq -\abs{a}\), and so \(a = 0\). Thus, \(\bigcap X = \{0\}\), as required. \\

            \((\impliedby)\) Let \(a \in \cal{G}\) with \(D(a) \neq \emptyset\). In particular, \(a \neq 0\), and so \(a \notin \bigcap X\). It follows that there is some \(\cal{J} \in X\) such that \(a \notin \cal{J}\). Hence, \(X \cap D(a) \neq \emptyset\), and so \(X\) is dense.
        \end{proof}

        \begin{theorem} [The Topological Representation Theorem]
            Let \(\cal{V} = (\cal{G},\cal{L},P)\) be a valued \(\ell\)-group, and:
            \[P_* : \spectrum{\cal{L}} \to \lspec{\cal{G}} : \cal{F} \mapsto (\nu_{\lspec{\cal{G}}}^{-1} \circ \spectrum{\epsilon_{\cal{V}}})(\cal{F})\]

            with \(X = \image(P_*)\). Then:
            \begin{itemize}
                \item[\((1)\)] \(X\) is a spectral subspace of \(\lspec{\cal{G}}\), and the map \(\Phi : \qcclat{X} \to \cal{L}\) given by \(\Phi = \mu_{\cal{L}}^{-1} \circ \qcclat{P_*}\) is a bounded lattice isomorphism such that \((\id{\cal{G}},\Phi)\) is an isomorphism of valued \(\ell\)-groups.

                \item[\((2)\)] If \(X = \lspec{\cal{G}}\), then \(\Phi= \epsilon_{\cal{V}}\).

                \item[\((3)\)] \(V(- \meet 0) \cap X\) is a dense valuation if and only if \(X\) is dense in \(\lspec{\cal{G}}\).
            \end{itemize}
        \end{theorem}
        \begin{proof}
            \((1)\) By Lemma \ref{NatValFuncExtensions}, we have that \(\epsilon_{\cal{V}}\) is a surjective bounded lattice morphism. Hence, by Spectral Duality, \(P_*\) is is a homeomorphism onto its image, and so, \(\Phi\) is a bounded lattice isomorphism. By \cite[Corollary \(2.1.4\)]{TheBible}, \(X\) is a spectral subspace. Now, it remains to check that \((\id{\cal{G}},\Phi)\) is an isomorphism of valued \(\ell\)-groups. In particular, we show that this diagram commutes:
            \begin{center}
                \begin{tikzcd}
                    & \cal{G} \arrow[rd, "P" description] \arrow[ld, "V(- \meet 0) \cap X" description] &         \\
                    \qcclat{X} \arrow[rr, "\Phi" description] &                                                                                   & \cal{L}
                \end{tikzcd}
            \end{center}

            First, let \(l \in \cal{L}\), \(C \in \qcclat{X}\), and notice that:
            \begin{align*}
                \Phi(C) = l \iff & \qcclat{P_*}(C) = V(l) \\
                \iff & (\forall \cal{F} \in \spectrum{\cal{L}})(\cal{F} \in \qcclat{P_*}(C) \iff \cal{F} \in V(l)) \\
                \iff & (\forall \cal{F} \in \spectrum{\cal{L}})(P_*(\cal{F}) \in C \iff l \in \cal{F}).
            \end{align*}

            Hence, let \(a \in \cal{G}\), and see that for \(\cal{F} \in \spectrum{\cal{L}}\):
            \begin{align*}
                P_*(\cal{F}) \in V(a \meet 0) \iff & V(a \meet 0) \in \spectrum{\epsilon_{\cal{V}}}(\cal{F}) \\ 
                \iff & \epsilon_{\cal{V}}(V(a \meet 0)) \in \cal{F} \\
                \iff & P(a) \in \cal{F}.
            \end{align*}

            Thus, \(\Phi(V(a \meet 0)) = P(a)\), and so the above diagram commutes as required. \\
            
            \((2)\) Let \(X = \lspec{\cal{G}}\). It is clear that \(\Phi(\bot) = \epsilon_{\cal{V}}(\bot)\). Then, for \(a \in \cal{G}\) and \(\cal{I} \in \spectrum{\cal{L}}\), we see that:
            \begin{align*}
                \cal{I} \in \qcclat{P_*}(V(a)) \iff & P_*(\cal{I}) \in V(a) \\
                \iff & a \in P_*(\cal{I}) \\
                \iff & V(a) \in \spectrum{\epsilon_{\cal{V}}}(\cal{I}) \\
                \iff & \epsilon_{\cal{V}}(V(a)) \in \cal{I} \\
                \iff & \cal{I} \in V(\epsilon_{\cal{V}}(V(a)).
            \end{align*}

            Hence, it follows that \(\Phi(V(a)) = \mu_{\cal{L}}^{-1}(V(\epsilon_{\cal{V}}(V(a))) = \epsilon_{\cal{V}}(V(a))\), as required. \\

            \((3)\) By Lemma \ref{DenseSubspaceAltDef}, we show that \(V(- \meet 0) \cap X\) is a dense valuation if and only if \(\bigcap X = \{0\}\). We check both implications separately:
            \begin{itemize}
                \item \((\implies)\) Let \(a \in \cal{G}\) with \(D(a) \neq \emptyset\). It is clear that \(a \neq 0\), and so \(a \notin \bigcap X\). Thus, there is some \(\cal{J} \in X\) with \(a \notin \cal{J}\), and so \(X \cap D(a) \neq \emptyset\). In particular, \(X\) is dense. 

                \item \((\impliedby)\) This is Lemma \ref{ValGrpsFromDenseSpecSubspaces}.
            \end{itemize}
        \end{proof}

        One quick consequence of the Topological Representation Theorem is the following alternative characterisation of the scaling invariance axiom. In particular, this brings the notion of a valuation on an \(\ell\)-groups closer in spirit to a valuation on a field.

        \begin{lemma} \label{ValuationAltDef}
            Let \(\cal{G}\) be an \(\ell\)-group, and \(\cal{L}\) a bounded distributive lattice. Let \(P : \cal{G} \to \cal{L}\) be a lattice morphism which is positivity affirming and essentially total. Then, \(P\) is scaling invariant if and only if, for all \(a,b \in \cal{G}\):
            \[P(a) \meeta P(b) \leqa P(a + b) \leqa P(a) \joina P(b).\]
        \end{lemma}
        \begin{proof}
            \((\implies)\) By assumption, \(\cal{V} \coloneqq (\cal{G},\cal{L},P)\) is a valued \(\ell\)-group. Hence, by the Topological Representation Theorem, let \(X \subseteq \lspec{\cal{G}}\) be a spectral subspace such that \(\cal{V} \cong (\cal{G},\qcclat{X},V(- \meet 0) \cap X)\). We assume WLOG that this isomorphism is equality. Now, for \(a,b \in \cal{G}\) and \(\cal{J} \in X\), we see that:
            \begin{align*}
                \cal{J} \in V(a \meet 0) \cap V(b \meet 0) \implies & a \meet 0, b \meet 0 \in \cal{J} \\
                \implies & \cal{J} \leq a + \cal{J}, b + \cal{J} \\
                \implies & \cal{J} \leq (a + b) + \cal{J} \\
                \implies & \cal{J} \in V((a + b) \meet 0) \ [\dagger]. 
            \end{align*}

            In particular, \(P(a) \meeta P(b) \leqa P(a + b)\). Now, as \(\quot{\cal{G}}{\cal{J}}\) is totally ordered, then we have that:
            \begin{align*}
                [\dagger] \implies & \cal{J} \leq (a + \cal{J}) + (b + \cal{J}) \\
                \implies & \cal{J} \leq 2 \cdot \max \{a + \cal{J}, b + \cal{J}\} \\
                \implies & \cal{J} \leq(2(a \join b)) + \cal{J} \\
                \implies & \cal{J} \in V((2(a \join b)) \meet 0). 
            \end{align*}

            In particular, \(P(a + b) \leqa P(2(a \join b)) = P(a \join b) = P(a) \joina P(b)\), as required. \\

            \((\impliedby)\) Let \(a \in \cal{G}\). We see immediately that:
            \[P(a) = P\left(\bigmeet_{i = 1}^n a\right) \leqa P(na) \leqa P\left(\bigjoin_{i = 1}^n a\right) = P(a).\]
        \end{proof}

%% file: Chapters/existing-literature.tex
    Existing literature has already covered the notion of a valuation of a lattice-ordered group. In particular, the 1997 paper \cite{ValuationsOfLatticeOrderedGroups} by Conrad, Darnel, and Nelson constructs a notion based on plenary subsets of regular \(\ell\)-ideals. In this section, we will briefly show that densely valued \(\ell\)-groups generalise their notion of valuation. First, we recall the definition of regular ideal and plenary set.

    \begin{define} [Regular Ideal]
        Let \(\cal{G}\) be an \(\ell\)-group, and \(\cal{J} \ideal \cal{G}\). Then, \(\cal{J}\) is \textbf{regular} if any of the following equivalent conditions hold:
        \begin{itemize}
            \item[\((1)\)] There exists some \(g \in \cal{G}\) such that \(\cal{J}\) is maximal amongst \(\ell\)-ideals not containing \(g\).

            \item[\((2)\)] There exists some \(g \in \cal{G}\) such that \(\cal{J} \in D(g)^{\max}\). In particular, regular \(\ell\)-ideals are prime.

            \item[\((3)\)] \(\cal{J}\) is a locally closed point\footnote{There exists \(U,C \subseteq \lspec{\cal{G}}\) open and closed respectively such that \(\{\cal{J}\} = U \cap C\).} in \(\lspec{\cal{G}}\).
        \end{itemize} 
        
        For \(g\) satisfying \((1)\) or \((2)\), we say that \(\cal{J}\) is a \textbf{value} of \(g\).
    \end{define}
    
    \begin{define} [Plenary Set]
        Let \(\cal{G}\) be an \(\ell\)-group, and \(\Pi \subseteq \lspec{\cal{G}}\). Then, \(\Pi\) is \textbf{plenary} if it is an up-set of regular \(\ell\)-ideals with \(\bigcap \Pi = \{0\}\).
    \end{define}

    The key notion used in the Conrad-Darnel-Nelson construction is that of the valuation lattice.

    \begin{define} [Valuation Lattice]
        Let \(\cal{G}\) be an \(\ell\)-group, \(\Pi \subseteq \lspec{\cal{G}}\) a plenary set, and for each \(g \in \cal{G}^+\):
        \[\Pi(g) \coloneqq \{\cal{J} \in \Pi \ | \ \cal{J} \text{ is a value of } g\} = \Pi \cap D(g)^{\max}.\]
        
        Then, the \textbf{valuation lattice} of \(\cal{G}\) is the set \(\Pi(\cal{G}) \coloneqq \{\Pi(g) \ | \ g \in \cal{G}^+\} \cup \{\infty\}\). We note this is a bounded distributive lattice with operations given for \(a,b \in \cal{G}^+\) by:
        \begin{align*}
            \Pi(a) \meeta \Pi(b) = & \Pi(a \meet b)  & \Pi(a) \joina \Pi(b) = & \Pi(a \join b) \\
            \top = & \infty & \bot = & \Pi(0) = \emptyset.
        \end{align*}

        The proof of this is non-trivial (see \cite[Proposition \(2.2\)]{ValuationsOfLatticeOrderedGroups} for a full proof).
    \end{define}

    We recall that, for a poset \(P\), we write \(\op{P}\) for the poset given by reversing the order on \(P\). For a lattice \(\cal{L}\), we remark \(\op{\cal{L}}\) for also a lattice, with meet and join reversed.

    \begin{lemma} \label{ValuationsAsValuedGroups}
        Let \(\cal{G}\) be an \(\ell\)-group, and \(\Pi \subseteq \lspec{\cal{G}}\) a plenary set. Let \(\cal{L} = \op{\Pi(\cal{G})}\) (where we use \(\op{(-)}\) to distinguish between operations in \(\Pi(\cal{G})\) and \(\op{\Pi(\cal{G})}\)), and \(P : \cal{G} \to \cal{L}\) such that \(P(g) = \Pi(g \meet 0)\). Then, \(P\) is a dense valuation on \(\cal{G}\).
    \end{lemma}
    \begin{proof}
        \((1)\) We check each property in turn:
        \begin{itemize}
            \item (Lattice Morphism) Let \(a,b \in \cal{G}\). Then, we see that, as \(\Pi(g) = \Pi(\abs{g})\) for all \(g \in \cal{G}\):
            \begin{align*}
                P(a \meet b) = & \Pi((a \meet b) \meet 0) \\
                = & \Pi(\abs{(a \meet b) \meet 0}) \\
                = & \Pi(((-a) \join 0) \join ((-b) \join 0)) \\
                = & \Pi((-a) \join 0) \joina \Pi((-b) \join 0) \\
                = & \Pi(a \meet 0) {\meeta}^{\norm{op}} \Pi(b \meet 0) \\
                = & P(a) {\meeta}^{\norm{op}} P(b).
            \end{align*}

            and similarly for \(\join\). Hence, \(P\) is a lattice morphism.

            \item (Scaling Invariant) Let \(a \in \cal{G}\) and \(n \in \bb{N}\). We notice then that:
            \begin{align*}
                P(a) = & \Pi(a \meet 0) \\
                = & \Pi \cap D(a \meet 0)^{\max} \\
                = & \Pi \cap D(n(a \meet 0))^{\max} \\
                = & \Pi \cap D((na) \meet 0)^{\max} \\
                = & \Pi((na) \meet 0) \\
                = & P(na).
            \end{align*}

            \item (Essentially Surjective) As \(\op{\bot} = \infty\), then let \(a \in \cal{G}^+\). It is clear that \(\Pi(a) = \Pi((-a) \meet 0) = P(-a)\), and so \(\cal{L} \setminus \{\bot\} \subseteq \image(P)\), as required.

            \item (Positivity Affirming and Detecting) First, let \(a \in \cal{G}^+\). It is clear that \(P(a) = \Pi(a \meet 0) = \Pi(0) = \op{\top}\). For the other direction, let \(a \in \cal{G}\) with \(P(a) = \op{\top}\). By definition, we have that \(a \meet 0 \in \bigcap \Pi = \{0\}\), and so \(0 \leq a\).
        \end{itemize}
    \end{proof}

    \begin{corollary}
        Let \(\cal{G}\) be an \(\ell\)-group, and \(\Pi \subseteq \lspec{\cal{G}}\) a plenary set. Let \(\tau\) be the topology on \(\lspec{\cal{G}}\), and \(\patch{\tau}\) the patch topology on \(\lspec{\cal{G}}\), and \(X = \Cls{\Pi}{\patch{\tau}}\). Then, \(X\) is a dense spectral subspace, and the map \(\alpha : \op{\Pi(\cal{G})} \to \qcclat{X}\) given by \(\alpha(\Pi(g)) \mapsto V(g) \cap X\) is a bounded lattice isomorphism, and induces a valued \(\ell\)-group isomorphism:
        \[(\id{\cal{G}},\alpha) : (\cal{G}, \op{\Pi(\cal{G})},\Pi(- \meet 0)) \to (\cal{G},\qcclat{X},V(- \meet 0) \cap X).\]
    \end{corollary}
    \begin{proof}
        First, we show that \(X\) is a dense spectral subspace of \(\lspec{\cal{G}}\). By \cite[Theorem \(2.1.3\)]{TheBible}, \(X\) is a spectral subspace. Then, as \(\bigcap X \subseteq \bigcap \Pi = \{0\}\), then \(X\) is dense by Lemma \ref{DenseSubspaceAltDef}. Now, we show that \(\alpha\) is a well-defined bounded lattice isomorphism. Hence, consider:
        \begin{itemize}
            \item (Well-Defined) Let \(g,h \in \cal{G}\) be such that \(\Pi(g) = \Pi(h)\), and suppose \(\alpha(\Pi(g)) \neq \alpha(\Pi(h))\). In particular (WLOG), let \(\cal{J} \in X\) be such that \(\cal{J} \in V(g) \cap D(h)\). As \(\cal{J} \in X = \Cls{\Pi}{\patch{\tau}}\), then \(V(g) \cap D(h) \cap \Pi \neq \emptyset\). Hence, let \(\cal{I} \in \Pi \cap V(g) \cap D(h)\). Then, there is some \(\cal{I}' \in D(h)^{\max}\) such that \(\cal{I} \subseteq \cal{I}'\). As \(\Pi\) is plenary, then \(\cal{I}' \in \Pi\). Now, we notice that as \(g \in \cal{I}\), then \(g \in \cal{I}'\). In particular, \(\cal{I}' \in \Pi(h)\) and \(\cal{I}' \notin \Pi(g)\) \Lightning. Thus, \(\alpha(\Pi(g)) = \alpha(\Pi(h))\), as required. 

            \item (Lattice Morphism) Let \(g,h \in \cal{G}^+\). We notice immediately that:
            \begin{align*}
                \alpha(\Pi(g) \op{{\joina}} \Pi(h)) = & \alpha(\Pi(g \meet h)) \\
                = & V(g \meet h) \cap X \\
                = & (V(g) \cap X) \cup (V(h) \cap X) \\
                = & \alpha(\Pi(g)) \cup \alpha(\Pi(h))
            \end{align*}

            and similarly for \(\op{{\meeta}}\).

            \item (Injective) Let \(g,h \in \cal{G}^+\) be such that \(\Pi(g) \neq \Pi(h)\). In particular, there is some \(\cal{I} \in \Pi\) such that (WLOG) \(g \in \cal{I}\) and \(h \notin \cal{I}\). Then, as \(\Pi \subseteq X\), we have that \(\alpha(\Pi(g)) = V(g) \cap X \neq V(h) \cap X = \alpha(\Pi(h))\).

            \item (Surjective) Let \(g \in \cal{G}^+\). It is immediate that \(\alpha(\Pi(g)) = V(g) \cap X\), as required. 
        \end{itemize}

        Now, we show that \((\id{\cal{G}},\alpha)\) is a valued \(\ell\)-group isomorphism. However, this follows immediately from the fact that \(V(a \meet 0) \cap X = \Pi(a \meet 0)\), and so we are done. 
    \end{proof}
    
    In particular, every Conrad-Darnel-Nelson valuations is associated to some densely valued \(\ell\)-groups. The following example illustrates that we can't improve this relationship, and hence that dense valuations in our sense strictly generalise Conrad-Darnel-Nelson valuations.

    \begin{example}
        Let \(\cal{G} = \bb{R} \times_{\norm{lex}} \bb{R}\) be the lexicographic product of \(\bb{R}\) with itself, and let \(\cal{J} = \{(a,0) \ | \ a \in \bb{R}\}\) so that \(\lspec{\cal{G}} = \{\{0\},\cal{J},\cal{G}\}\). It is clear that \(X = \{\{0\}\}\) is a spectral subspace of \(\lspec{\cal{G}}\}\), and so by \cite[Corollary \(4.4.6\)]{TheBible} that is dense in \(\lspec{\cal{G}}\}\). However, we notice that \(\cal{J}\) is maximal amongst \(\ell\)-ideals not containing \((0,1)\), and so is regular. Hence, \(X\) isn't plenary. 
    \end{example}

%% file: Chapters/standard-structures.tex
    \subsection{Functional Representations of Densely Valued \(\ell\)-Groups}

        The study of \(\ell\)-groups is intimately related to the study of function spaces of the form \(\Gamma^X\) for \(\Gamma\) a DOAG, as in particular every \(\ell\)-group embeds in one of this form. The following section demonstrates that a similar result holds for densely valued \(\ell\)-groups - namely, we can always embed a densely valued \(\ell\)-group \(\cal{V}\) into a \textbf{standard structure} (Definition \ref{StanStructDef}). 

        \begin{lemma} \label{FuncRepsForValGrps}
            Let \(\cal{V} = (\cal{G},\cal{L},P)\) be a densely valued \(\ell\)-group. Then, there is a DOAG \(\Gamma\) and \(\ell\)-group embedding \(\iota : \cal{G} \to \Gamma^{\spectrum{\cal{L}}}\) such that, for \(a \in \cal{G}\) and \(\cal{F} \in \spectrum{\cal{L}}\), \(\iota(a)(\cal{F}) \geq 0\) if and only if \(P(a) \in \cal{F}\).
        \end{lemma}
        \begin{proof}
            By the Topological Representation Theorem, let \(X \subseteq \lspec{\cal{G}}\) be a dense spectral subspace such that \(\cal{V} \cong (\cal{G},\qcclat{X},V(- \meet 0) \cap X)\). We can assume WLOG that this isomorphism is equality. Now, we define a map:
            \[I_X : \cal{G} \to \prod_{\cal{J} \in X} \quot{\cal{G}}{\cal{J}} : a \mapsto (a + \cal{J})_{\cal{J} \in X}.\]

            Clearly, \(I_X\) is an \(\ell\)-group morphism. Further, we note that for \(a \in \cal{G}\):
            \[I_X(a) = 0 \iff  (\forall \cal{J} \in X)(a \in \cal{J}) \iff a \in \bigcap X \iff a = 0.\]

            Thus, \(\ker(I_X) = \{0\}\), and so \(I_X\) is an \(\ell\)-group embedding. Further, for \(a \in \cal{G}\) and \(\cal{J} \in X\), we see that \(I_X(a)(\cal{J}) = 0\) if and only if \(a + \cal{J} = \cal{J}\) if and only if \(a \in \cal{J}\). Then, as \(\nu_X\) is an isomorphism of spectral spaces, then we have an isomorphism:
            \[\nu_X^* : \prod_{\cal{J} \in X} \quot{\cal{G}}{\cal{J}} \to \prod_{\cal{F} \in \spectrum{\qcclat{X}}} \quot{\cal{G}}{\nu_X^{-1}(\cal{F})} : (a + \cal{J})_{\cal{J} \in X} \mapsto (a + \nu_X^{-1}(\nu_X(\cal{J})))_{\cal{J} \in X}.\]

            We then see that, for \(a \in \cal{G}\) and \(\cal{F} \in \spectrum{\qcclat{X}}\), we have:
            \begin{align*}
                0 \leq (\nu_X^* \circ I_X)(a)(\cal{F}) \iff & \nu_X^{-1}(\cal{F}) \leq a + \nu_X^{-1}(\cal{F}) \\
                \iff & a \meet 0 \in \nu_X^{-1}(\cal{F}) \\
                \iff & \nu_X^{-1}(\cal{F}) \in V(a \meet 0) \cap X \\
                \iff & V(a \meet 0) \cap X \in \cal{F}.
            \end{align*}

            Finally, by taking divisible hulls in each component and we have that \(\cal{G}\) embeds into a product of DOAGs. By \cite[Exercise \(2.8.10\)]{Hodges}, the theory of DOAGs is complete, and so there is a DOAG \(\Gamma\) with \(\quot{\cal{G}}{\cal{J}} \subgrp \Gamma\) for each \(\cal{J} \in X\). Hence, there is a valued \(\ell\)-group embedding \(\iota : \cal{G} \to \Gamma^{\spectrum{\cal{L}}}\) which satisfies \(\iota(a)(\cal{F}) \geq 0 \iff P(a) \in \cal{F}\).
        \end{proof}

        We note that, by identifying \(\cal{G}\) with \(\iota(\cal{G})\), we can always assume that \(\iota\) is the inclusion map.

        As alluded to above, this allows us to define an important class of densely valued \(\ell\)-groups - the standard structures.

        \begin{lemma} [Standard Valuations\label{StanStructDef}]
            Let \(\Gamma\) be a divisible \(o\)-group and \(X\) a non-empty set. Then, the map:
            \[\{- \geq 0\}: \Gamma^X \to \cal{P}(X) : f \mapsto \{f \geq 0\}\]

            where \(\{f \geq 0\} = \{x \in X \ | \ f(x) \geq 0\}\) is a valuation on \(\Gamma^X\), and surjective. We call this the \textbf{standard valuation} on \(\Gamma^X\), and write \(\stan{\Gamma^X}\) for the triple \((\Gamma^X,\cal{P}(X),\{- \geq 0\})\), which we call the \textbf{standard structure} associated to \(\Gamma^X\).
        \end{lemma}
        \begin{proof}
            Immediate from the motivation for the definition of a dense valuation, noting that for any \(S \subseteq X\), the map:
            \[f : X \to \Gamma : x \mapsto \begin{cases}
                g & x \in S \\
                -g & x \notin S
            \end{cases}\]

            is such that \(S = \{f \geq 0\}\) for any choice of \(g \in \Gamma^+\).
        \end{proof}

        The importance of standard structures follows from the next result, which tells us that every densely valued \(\ell\)-group can be viewed as a substructure of a standard structure. In particular, this result generalises the classical result for \(\ell\)-groups, by letting \(\cal{V} = \val{G}\).

        \begin{theorem} [The Functional Representation Theorem]
            Let \(\cal{V} = (\cal{G},\cal{L},P)\) be a densely valued \(\ell\)-group. Then, there is a divisible \(o\)-group \(\Gamma\) such that \(\cal{V} \subgrp \stan{\Gamma^{\spectrum{\cal{L}}}}\).
        \end{theorem}
        \begin{proof}
            By Lemma \ref{FuncRepsForValGrps}, we have an \(\ell\)-group embedding \(I : \cal{G} \to \Gamma^{\spectrum{\cal{L}}}\) for some divisible \(o\)-group \(\Gamma\) such that, for all \(a \in \cal{G}\) and \(\cal{F} \in \spectrum{\cal{L}}\), \(I(a)(\cal{F})\geq  0\) if and only if \(P(a) \in \cal{F}\). Now, let \(X = \spectrum{\cal{L}}\), and define a map:
            \[J : \cal{L} \to \{\{f \geq 0\} \ | \ f \in \Gamma^X\} : \begin{cases}
                P(a) \mapsto \{I(a) \geq 0\}\\
                \bot \mapsto \emptyset.
            \end{cases}\]

            We claim that \(J\) is a well-defined bounded lattice embedding. Hence, consider:
            \begin{itemize}
                \item (Well-Defined and Injective) Let \(a,b \in \cal{G}\). We see that:
                \begin{align*}
                    P(a) = P(b) \iff & (\forall \cal{F} \in X)(P(a) \in \cal{F} \iff P(b) \in \cal{F}) \\
                    \iff & (\forall \cal{F} \in X)(I(a)(\cal{F}) \geq 0 \iff I(b)(\cal{F}) \geq 0) \\
                    \iff & J(P(a)) = J(P(b)).
                \end{align*}

                Hence, \(J\) is well-defined and injective.

                \item (Meets and Joins) We see that, for \(a,b \in \cal{G}\):
                \begin{align*}
                    J(P(a) \meeta P(b)) = & J(P(a \meet b)) \\
                    = & \{I(a \meet b) \geq 0\} \\
                    = & \{I(a) \meet I(b) \geq 0\} \\
                    = & \{I(a) \geq 0\} \cap \{I(b) \geq 0\} \\
                    = & J(P(a)) \cap J(P(b))
                \end{align*}

                and \(\joina\) follows similarly.

                \item (Bounds) \(J(\bot) = \emptyset\) by definition. Further, we see that \(J(\top) = J(P(0)) = \{I(0) \geq 0\} = X\). Hence, \(J\) preserves bounds.
            \end{itemize}

            Finally, notice that by definition for \(a \in \cal{G}\), \((\{- \geq 0\} \circ I)(a) = \{I(a) \geq 0\} = (J \circ P)(a)\). Hence, \((I,J)\) is a valued \(\ell\)-group embedding, as required.
        \end{proof}

        \begin{remark}
            For \(\Gamma\) a DOAG, \(X\) a set, and \(f \in \Gamma^X\), we define \(\{f = 0\} \coloneqq \{f \geq 0\} \cap \{(-f) \geq 0\}\), and remark that \(x \in \{f = 0\}\) if and only if \(f(x) = 0\).
        \end{remark}

        The Functional Representation Theorem is an important result for two reasons. First, it allows us to view densely valued \(\ell\)-groups as groups of functions into a DOAG. This is much easier to reason about, and gives us a method for producing future results. 

        The second reason is it affords us access to the Shen-Weispfenning theorem - a seminal result which gives partial quantifier elimination in the case for the group sort of a densely valued \(\ell\)-group is divisible and satisfies a condition called patching, we shall discuss in the next section.

    \subsection{The Shen-Weispfenning Theorem}

        The Shen-Weispfenning theorem, due to \cite{ShenWeispfenning}, gives an effective\footnote{In the sense of computability theory.} procedure for reducing \(\lvlgrp\)-formulae to \(\lbdlat\)-formulae on the lattice sort for a certain class of densely valued \(\ell\)-groups. The original formulation is in the form of what has been termed ``standard structures" in other literature. These differ a priori from our notion of standard structure, and so we shall call them \textbf{Shen-Weispfenning (SW) standard structures}:

        \begin{define} [SW Standard Structure]
            Let \(\cal{G}\) be an \(\ell\)-group, \(\Gamma\) a DOAG, and \(X\) a set with an \(\ell\)-group embedding \(I : \cal{G} \to \Gamma^X\). Then, the \textbf{SW standard structure} of \((\cal{G},I)\) is the triple \((\cal{G},\cal{L}_I,\{- \geq 0\}\), where \(\cal{L}_I = \{\{I(f) \geq 0\} \ | \ f \in \cal{G}\}\). We always consider \((\cal{G},\cal{L}_I,\{- \geq 0\})\) as a structure in the language \(\lvlgrp\), and remark that it is a densely valued \(\ell\)-group. \cite[Def. \(2.4.15\)]{RicardoThesis}.
        \end{define}

        As alluded to, the Shen-Weispfenning requires that our SW standard structure has a property called patching. This is reminiscent of the Tiestz Extension Theorem for real-valued functions - it says that, if two functions \(f\) and \(g\) agree on the overlap of two closed sets \(C\) and \(D\), then there is a function \(h\) which agrees with \(f\) on \(C\) and \(g\) on \(D\).

        \begin{define} [Patching for SW Standard Structures]
            Let \(\cal{G}\) be an \(\ell\)-group, \(\Gamma\) a DOAG, and \(X\) a set with an \(\ell\)-group embedding \(I : \cal{G} \to \Gamma^X\). Then, \((\cal{G},\cal{L}_I,\{- \geq 0\})\) has \textbf{SW patching} if, for all \(f,g,\phi,\psi \in \cal{G}\) with \(\{I(f) \geq 0\} \cap \{I(g) \geq 0\} \subseteq \{\phi = \psi\}\), there exists \(\chi \in \cal{G}\) such that \(\{I(f) \geq 0\} \subseteq \{\chi = \phi\}\) and \(\{I(g) \geq 0\} \subseteq \{\chi = \psi\}\). \cite[Def. \(2.4.19\)]{RicardoThesis}.
        \end{define}

        \begin{theorem} [The Shen-Weispfenning Theorem for SW Standard Structures]
            For every \(\lvlgrp\)-formula \(\phi(\bar{v},\bar{w})\) of sort \(\bb{G}^n \times \bb{L}^m\), there exists:
            \begin{itemize}
                \item An \(\lvlgrp\)-formula \(\chi(\bar{p},\bar{w})\) of sort \(\bb{L}^{k + m}\) for some \(k \in \bb{N}\); and

                \item For each \(i \in \{1,\ldots,k\}\), an \(\lvlgrp\)-term \(t_i(\bar{v}) : \bb{G}^n \to \bb{G}\)
            \end{itemize}

            such that, for all divisible \(\ell\)-groups \(\cal{G}\), DOAGs \(\Gamma\) and sets \(X\) with an \(\ell\)-group embedding \(I : \cal{G} \to \Gamma^X\) such that \((\cal{G},\cal{L}_I, \{- \geq 0\})\) has SW patching, the \(\lvlgrp\)-formula:

            \[\phi_{\norm{reduct}}(\bar{v},\bar{w}) \coloneqq (\exists p_1,\ldots,p_k : \bb{L})(\chi(\bar{p},\bar{w}) \land \bigland_{i = 1}^k (p_i = P(t_i(\bar{v}))))\]

            satisfies \((\cal{G},\cal{L}_I,\{- \geq 0\}) \models (\forall \bar{v} : \bb{G})(\forall \bar{w} : \bb{L})(\phi(\bar{v},\bar{w}) \leftrightarrow \phi_{\norm{reduct}}(\bar{v},\bar{w}))\). 
            
            Further, for \(\phi(\bar{v},w)\) (positive) existential, \(\phi_{\norm{reduct}}\) is also (positive) existential; and both \(\chi(\bar{p},\bar{w})\) and the \(t_i(\bar{v})\)'s can be effectively obtained from \(\phi(\bar{v},\bar{w})\).
        \end{theorem}
        
        Due to the proof's complexity, we choose to omit proving Shen-Weispfenning (see \cite[Section \(2.4.3\)]{RicardoThesis} or \cite[Section \(3\)]{MarcusShenWeisPaper} for a full proof). However, due to its importance, we provide a rough sketch, and highlight where our assumptions are used. There are \(4\) major steps to the proof:
        \begin{itemize}
            \item[\((1)\)] We show that \(\phi(\bar{v},\bar{w})\) is equivalent to an \(\lvlgrp\)-formula \(\phi_*(\bar{v},\bar{w})\) containing neither \(\meet\) nor \(\join\). We first do this for when \(\cal{G}\) is totally-ordered, and then in general using the representation of \(\cal{G}\) as a subdirect product of totally-ordered \(\ell\)-groups.

            \item[\((2)\)] We prove Shen-Weispfenning in the case where \(\phi_*(\bar{v},\bar{w})\) contains no quantifiers in \(\cal{G}\), by noting that density gives us \(a \leq b \iff 0 \leq b - a \iff P(b - a) = \top\) for any \(a,b \in \cal{G}\).

            \item[\((3)\)] Using patching, we reduce \(\lvlgrp\)-formula \(\psi(v,w)\) of sort \(\bb{G} \times \bb{L}\) of the form \((\exists a : \bb{G})(y \leqa P(b - x)\) (and permutations thereof) to ones which don't quantify over \(\cal{G}\).

            \item[\((4)\)] By applying these reductions, we can induct on quantifier complexity of \(\phi_*(\bar{v},\bar{w})\) to show the full result. We use the divisibility implicitly in this step, in order to reduce \(\lvlgrp\)-terms of the form \(P(nv - mw)\) to ones of the form \(P(v' - w')\).
        \end{itemize}

        It is clear that Shen-Weispfenning is significant - it computably reduces \(\lvlgrp\)-formulae to formulae over the lattice sort. In order to utilise it, it is sufficient to show that our notion of standard structure coincides with the SW standard structure. 

        \begin{lemma} \label{ValLGrpsAreSWStanStrucs}
            Let \(\cal{V} = (\cal{G},\cal{L},P)\) be a densely valued \(\ell\)-group. Then \(\cal{V}\) is isomorphic to an SW standard structure.
        \end{lemma}
        \begin{proof}
            By the Functional Representation Theorem, there is a DOAG \(\Gamma\) with a valued \(\ell\)-group embedding \((I,J) : \cal{V} \to \stan{\Gamma^{\spectrum{\cal{L}}}}\). In particular, \(I : \cal{G} \to \Gamma^{\spectrum{\cal{L}}}\) is an \(\ell\)-group embedding of \(\cal{G}\) into \(\Gamma^X\) for some DOAG \(\Gamma\) and set \(X\). Now, let \(\cal{L}_I = \image(J)\). This is clearly a bounded lattice isomorphism, and we see by definition that, for \(f \in \cal{G}\), \(\{I(f) \geq 0\} = J(P(f))\). Thus, it is immediate that \(\cal{V} \cong (\cal{G},\cal{L}_I,\{- \geq 0\})\), and this is precisely an SW standard structure.
        \end{proof}

        \begin{corollary}
            The \(\lvlgrp\)-classes of SW standard structures and densely valued \(\ell\)-groups coincide.
        \end{corollary}
        \begin{proof}
            Immediate from Lemma \ref{ValLGrpsAreSWStanStrucs}.
        \end{proof}

        Thus, by the above observation, we are able to restate Shen-Weispfenning in the context of valued \(\ell\)-groups.

        \begin{define} [Patching\label{PatchingDef}]
            Let \(\cal{V} = (\cal{G},\cal{L},P)\) be a valued \(\ell\)-group. Then, \(\cal{V}\) has \textbf{patching} if, for all \(f,g,\phi,\psi \in \cal{G}\) with \(P(f) \meeta P(g) \leqa P(\phi - \psi) \meeta P(\psi - \phi)\), there exists \(\chi \in \cal{G}\) such that \(P(f) \leqa P(\phi - \chi) \meeta P(\chi - \phi)\) and \(P(g) \leqa P(\psi - \chi) \meeta P(\chi - \psi)\).
        \end{define}

        \begin{theorem} [The Shen-Weispfenning Theorem\label{ShenWeispfenning}]
            Let \(\tvlgrp^+\) be the theory of valued \(\ell\)-groups with divisible group sort and patching. Then, for every \(\lvlgrp\)-formula \(\phi(\bar{v},\bar{w})\) of sort \(\bb{G}^n \times \bb{L}^m\), there exists:
            \begin{itemize}
                \item An \(\lvlgrp\)-formula \(\chi(\bar{p},\bar{w})\) of sort \(\bb{L}^{k + m}\) for some \(k \in \bb{N}\); and

                \item For each \(i \in \{1,\ldots,k\}\), an \(\lvlgrp\)-term \(t_i(\bar{v}) : \bb{G}^n \to \bb{G}\)
            \end{itemize}

            such that the \(\lvlgrp\)-formula:
            \[\phi_{\norm{reduct}}(\bar{v},\bar{x}) \coloneqq (\exists y_1,\ldots,y_k : \bb{L})(\chi(\bar{y},\bar{x}) \land \bigland_{i = 1}^k (y_i = P(t_i(\bar{a}))))\] 
            
            satisfies \(\tvlgrp^+ \entails (\forall \bar{a} : \bb{G})(\forall \bar{x} : \bb{L})(\phi(\bar{a},\bar{x}) \leftrightarrow \phi_{\norm{reduct}}(\bar{a},\bar{x}))\).

            Further, for \(\phi(\bar{v},w)\) (positive) existential, \(\phi_{\norm{reduct}}\) is also (positive) existential; and both \(\chi(\bar{p},\bar{w})\) and the \(t_i(\bar{v})\)'s can be effectively obtained from \(\phi(\bar{v},\bar{w})\).
        \end{theorem}
        \begin{proof}
            This is just a translation of the Shen-Weispfenning theorem for SW standard structures.
        \end{proof}

        An important fact about standard structures is that they have patching and a divisible group sort. That \(\Gamma^X\) is divisible for \(\Gamma\) a DOAG is immediate, and the following lemma establishes \(\stan{\Gamma^X}\) also has patching.

        \begin{lemma} \label{StanStructsHavePatching}
            Let \(\Gamma\) be an \(o\)-group, and \(X\) a non-empty set. Then, \(\stan{\Gamma^X}\) has patching.
        \end{lemma}
        \begin{proof}
            Let \(f,g \in \Gamma^X\). It is clear that \(\{f - g \geq 0\} \cap \{g - f \geq 0\} = \{f = g\}\). Hence, for \(C,D \in \cal{L}\) with \(C \cap D \subseteq \{f = g\}\), we define a map:
            \[h : X \to \Gamma : x \mapsto \begin{cases}
                f(x) & x \in C \\
                g(x) & x \in D \\
                0 & \text{otherwise.}
            \end{cases}\]

            As \(f(x) = g(x)\) for all \(x \in C \cap D\), then this is well-defined. Thus, for \(x \in C\), we have that \(f(x) = h(x)\), and for \(x \in D\), \(g(x) = h(x)\). In particular, \(C \subseteq \{f = h\}\) and \(D \subseteq \{g = h\}\), and so \(\stan{\Gamma^X}\) has patching, as required. 
        \end{proof}

    \subsection{Algebraically Closed Densely Valued \(\ell\)-Groups}

        Recall that, for a class of \(\cal{L}\)-structures \(\bb{K}\), an \(\cal{L}\)-structure \(\cal{M} \in \bb{K}\) is algebraically closed in \(\bb{K}\) if, for all \textit{positive} existential \(\cal{L}\)-formulae \(\phi(\bar{v})\) of sort \(S_1 \times \ldots \times S_n\), and \(\cal{N} \in \bb{K}\) with \(\cal{M} \subgrp \cal{N}\), we have that:
        \[\cal{N} \models (\exists \bar{v} : S_1 \times \ldots \times S_n)(\phi(\bar{v})) \implies \cal{M} \models (\exists \bar{v} : S_1 \times \ldots \times S_n)(\phi(\bar{v})).\]
        
        We are now in a position to characterise the algebraically closed densely valued \(\ell\)-groups. In particular, we will make use of the fact that every densely valued \(\ell\)-group embeds into a standard structure, and that these satisfy the conditions of the Shen-Weispfenning theorem.

        \begin{theorem} [The Algebraic Characterisation of Algebraically Closed Densely Valued \(\ell\)-Groups\label{ACValGrpAltDef}]
            Let \(\cal{V} = (\cal{G},\cal{L},P)\) be a densely valued \(\ell\)-group. Then, \(\cal{V}\) is algebraically closed if and only if \(\cal{G}\) is divisible, \(\cal{L}\) is Boolean, and \(\cal{V}\) has patching.
        \end{theorem}
        \begin{proof}
            \((\implies)\) By the Functional Representation Theorem, let \(\Gamma\) be a divisible \(o\)-group such that \(\cal{V} \subgrp \stan{\Gamma^{\spectrum{\cal{L}}}}\). Now, we define \(\lvlgrp\)-formulae:
            \begin{align*}
                \phi_n(v) \equiv & (\exists a : \bb{G})(na = v) \\
                \psi(w) \equiv & (\exists x : \bb{L})(w \joina x = \top \land w \meeta x = \bot) \\
                \chi(v_1,v_2,w_1,w_2) \equiv & (\exists a : \bb{G})(w_1 \leqa P(a - v_1) \meeta P(v_1 - a) \land w_2 \leqa P(a - v_2) \meeta P(v_2 - a)).
            \end{align*}

            where \(n \in \bb{N}_{> 0}\), \(\phi_n(v)\) has sort \(\bb{G}\), \(\psi(w)\) has sort \(\bb{L}\), and \(\chi(v_1,v_2,w_1,w_2)\) has sort \(\bb{G}^2 \times \bb{L}^2\). We observe that these are all positive existential \(\lvlgrp\)-formulae. Then, it is clear that for all \(a \in \cal{G}\) and \(l \in \cal{L}\), we have:
            \[\stan{\Gamma^X} \models \phi_n(a), \psi(l) \implies \cal{V} \models \phi_n(a), \psi(l).\]

            In particular, \(\cal{G}\) is divisible, and \(\cal{L}\) is Boolean. Finally, if \(a,b \in \cal{G}\) and \(l,k \in \cal{L}\) satisfy \(l \meeta k \leqa P(a - b) \meeta P(b - a)\), then we have:
            \[\stan{\Gamma^X} \models \chi(a,b,l,k) \implies \cal{V} \models \chi(a,b,l,k)\]

            and so \(\cal{V}\) has patching. \\
    
            \((\impliedby)\) Let \(\phi(\bar{v},\bar{w})\) be a positive existential \(\lvlgrp\)-formula of sort \(\bb{G}^n \times \bb{L}^m\), \(\bar{a} \in \cal{G}^n\), and \(\bar{l} \in \cal{L}^m\). Let \(\cal{W} = (\cal{H},\cal{K},Q)\) be a densely valued \(\ell\)-group with \(\cal{V} \subgrp \cal{W}\) and \(\cal{W} \models \phi(\bar{a},\bar{x})\). By the Functional Representation Theorem, there is some divisible \(o\)-group \(\Gamma\) such that \(\cal{W} \subgrp \stan{\Gamma^{\spectrum{\cal{K}}}} \eqqcolon \widehat{\cal{W}}\). In particular, \(\widehat{\cal{W}} \models \phi(\bar{a},\bar{l})\). Hence, by the Shen-Weispfenning theorem, let:
            \begin{itemize}
                \item \(\chi(\bar{p},\bar{w})\) be an \(\lvlgrp\)-formula of sort \(\bb{L}^{k + m}\); and
    
                \item For each \(i \in \{1,\ldots,k\}\), \(t_i(\bar{v}) : \bb{G}^n \to \bb{G}\) an \(\lvlgrp\)-term
            \end{itemize}
    
            such that the \(\lvlgrp\)-formula:
            \[\phi_{\norm{reduct}}(\bar{v},\bar{w}) \equiv (\exists y_1,\ldots,y_k : \bb{L})(\chi(\bar{y},\bar{w}) \land \bigland_{i = 1}^k (y_i = P(t_i(\bar{v}))))\]
    
            is positive existential, and equivalent to \(\phi(\bar{v},\bar{w})\) modulo \(\tvlgrp^+\). In particular, we have that:
            \[(\exists y_1,\ldots,y_k \in \cal{P}(\spectrum{\cal{K}}))(\hat{\cal{W}} \models \chi(\bar{y},\bar{l}) \text{ and } (\forall i \in \{1,\ldots,k\})(y_i = \hat{Q}(t_i^{\hat{\cal{W}}}(\bar{a})))).\]
    
            However, we note that for \(1 \leq i \leq k\), \(y_i = \hat{Q}(t_i^{\hat{\cal{W}}}(\bar{a})) = P(t_i^{\cal{V}}(\bar{a}))\). Hence, as \(\cal{L}\) is Boolean, and \(\chi(\bar{p},\bar{w})\) is a positive existential \(\lvlgrp\)-formula of sort \(\bb{L}^{k + m}\), then it is immediate from \cite[Theorem \(6\)]{AlgClosedAndExtClosedDistroLats} that \(\cal{V} \models \chi(y_1,\ldots,y_k,\bar{l})\). Hence, we have that:
            \[\cal{V} \models \phi_{\norm{reduct}}(\bar{a},\bar{k})) \implies \cal{V} \models \phi(\bar{a},\bar{l})\]
    
            and so \(\cal{V}\) is algebraically closed. 
        \end{proof}

        It is a clear corollary that there is an \(\lvlgrp\)-theory of algebraically closed densely valued \(\ell\)-groups, which we denote by \(\tvlgrpac\). We now document a number of immediate properties of algebraically closed densely valued \(\ell\)-groups, including some topological consequences.

        \begin{lemma} \label{ACValGrpProps}
            Let \(\cal{V} = (\cal{G},\cal{L},P)\) be an algebraically closed densely valued \(\ell\)-group. Then:
            \begin{itemize}
                \item[\((1)\)] \(P\) is surjective, and so in particular \(\cal{G} \not\cong \{0\}\).

                \item[\((2)\)] \(\cal{G}\) is algebraically closed.
                
                \item[\((3)\)] \(X \coloneqq \lspec{\cal{G}}^{\min}\) is quasicompact, and \(X \cong \spectrum{\cal{L}}\).

                \item[\((4)\)] \(\cal{G}\) is \textbf{complemented} - for all \(a \in \cal{G}^+\), there exists \(b \in \cal{G}^+\) such that \(a \meet b = 0\) and \(a \join b\) is a weak order unit.

                \item[\((5)\)] For all \(\nu \in \cal{G}^+\), \(\nu\) is a weak order unit if and only if \(P(-\nu) = \bot\).

                \item[\((6)\)] \(\cal{G}\) is hyper-Archimedean if and only if, for all \(a \in \cal{G}\), \(\minideal{a} = \ppol{a}\). 
            \end{itemize}
        \end{lemma}
        \begin{proof}
            \((1)\) If \(\cal{G} \cong \{0\}\), then we remark that \(\Gamma = \bb{Q}\) is such that \(\cal{V} \subgrp \stan{\Gamma^{\spectrum{\cal{L}}}}\). Else, by the Functional Representation Theorem, let \(\Gamma\) be a divisible \(o\)-group such that \(\cal{V} \subgrp \stan{\Gamma^{\spectrum{\cal{L}}}}\). In either case, as \(\{- \geq 0\}\) is surjective, then for all \(l \in \cal{L}\):
            \[\stan{\Gamma^{\spectrum{\cal{L}}}} \models (\exists a : \bb{G})(P(a) = l) \implies \cal{V} \models (\exists a : \bb{G})(P(a) = l).\]

            \((2)\) As \(\cal{G}\) is divisible, it is enough to show that \(\cal{V} \models \alpha\), where:
            \begin{align*}
                \alpha_1(v_1,v_2,v_3) \equiv & (0 \leq v_3) \land (v_1 \meet v_2 = 0) \\
                \alpha_2(v_1,v_2,v_3,v_4,v_5) \equiv & (0 \leq v_4,v_5) \land (v_4 + v_5 = v_3) \land (v_1 \meet v_5 = v_4 \meet v_2 = v_4 \meet v_5 = 0)) \\
                \alpha \equiv & (\forall a,b,c : \bb{G})(\alpha_1(a,b,c) \rightarrow (\exists f,g : \bb{G})(\alpha_2(a,b,c,f,g))
            \end{align*}

            where \(\alpha_1(v_1,v_2,v_3)\) has sort \(\bb{G}^3\), and \(\alpha_2(v_1,v_2,v_3,v_4,v_5)\) has sort \(\bb{G}^5\). By the Functional Representation Theorem, let \(\Gamma\) be a divisible \(o\)-group such that \(\cal{V} \subgrp \stan{\Gamma^{\spectrum{\cal{L}}}}\). Let \(a,b,c \in \cal{G}\) such that \(0 \leq c\) and \(a \meet b = 0\). We define maps:
            \begin{align*}
                f : & \spectrum{\cal{L}} \to \Gamma : \cal{F} \mapsto \begin{cases}
                    c(\cal{F}) & a(\cal{F}) \neq 0 \\
                    0 & a(\cal{F}) = 0, 
                \end{cases} \\
                g : & \spectrum{\cal{L}} \to \Gamma : \cal{F} \mapsto \begin{cases}
                    0 & a(\cal{F}) \neq 0 \\
                    c(\cal{F}) & a(\cal{F}) = 0.
                \end{cases} 
            \end{align*}

            We claim that \(\stan{\Gamma^{\spectrum{\cal{L}}}} \models \alpha(a,b,c,f,g)\). Clearly, \(0 \leq f,g\), \(f + g = c\), and \(f \meet g = 0\). Further, we see that for \(\cal{F} \in \spectrum{\cal{L}}\):
            \begin{align*}
                (a \meet g)(\cal{F}) = & \min\{a(\cal{F}),g(\cal{F})\} = 0 \\
                (f \meet b)(\cal{F}) = & \min \{f(\cal{F}),b(\cal{F})\} = 0
            \end{align*}

            as \(a(\cal{F}) \neq 0 \implies b(\cal{F}) = 0\). Thus, as \(\cal{V}\) is algebraically closed:
            \[\stan{\Gamma^{\spectrum{\cal{L}}}} \models (\exists f,g : \bb{G})(\alpha_2(a,b,c,f,g)) \implies \cal{V} \models (\exists f,g : \bb{G})(\alpha_2(a,b,c,f,g))\]

            and so \(\cal{V} \models \alpha\), as required. \\
            
            \((3)\) By the Topological Representation Theorem, there is some dense spectral subspace \(Y \subseteq \lspec{\cal{G}}\) such that \(Y \cong \spectrum{\cal{L}}\). As \(\cal{L}\) is Boolean, then \(Y\) is a Boolean space. However, this means \(Y\) has no non-trivial specialisations, and by \cite[Corollary \(4.4.6\)]{TheBible}, \(X \subseteq Y\). Hence, it is clear that \(X = Y\), as required. \\

            \((4)\) Immediate from \((3)\) and \cite[Theorem \(2.2\)]{CompLGrps}. \\

            \((5)\) By \((1)\), \(\cal{V} \cong (\cal{G},\qcclat{X},V(- \meet 0) \cap X)\) where \(X = \lspec{\cal{G}}^{\min}\). Hence, let \(\nu \in \cal{G}\). Unpacking the definition, we see that:
            \begin{align*}
                P(-\nu) = \bot \iff & V((- \nu) \meet 0) \cap X= \emptyset \\
                \iff & V(\nu) \cap X = \emptyset \\
                \iff & X \subseteq D(\nu) \\
                \iff & D(\nu) \text{ is dense in } \lspec{\cal{G}} \\
                \iff & (\forall a \in \cal{G} \text{ non-zero})(D(\nu) \cap D(a) \neq D(0)) \\
                \iff & (\forall a \in \cal{G} \text{ non-zero})(D(\nu \meet a) \neq D(0) \\ 
                \iff & (\forall a \in \cal{G} \text{ non-zero})(0 \neq \nu \meet a) \\
                \iff & (\forall a \in \cal{G})(\nu \meet a = 0 \implies a = 0).
            \end{align*}

            \((6)\) We note that:
            \begin{itemize}
                \item \(\cal{G}\) is hyper-Archimedean if and only if \(\lspec{\cal{G}}^{\min} = \plspec{\cal{G}}\).

                \item By \cite[Thm \(3.1\)]{MartinezLGrps}, \(\minideal{g} = \ppol{g}\) for all \(g \in \cal{G}\) if and only if \(\lspec{\cal{G}}^{\min} \cup \{\cal{G}\}\) is dense in \(\patch{\lspec{\cal{G}}}\)
            \end{itemize}

            By \((1)\) and \cite[Proposition \(4.4.16\)]{TheBible}, \(\lspec{\cal{G}}^{\min} \cup \{\cal{G}\}\) is closed in \(\patch{\lspec{\cal{G}}}\), and so the equivalence is immediate.
        \end{proof}

        \begin{lemma} \label{StanStructsAreAC}
            Let \(\Gamma\) be a DOAG and \(X\) a set. Then, \(\stan{\Gamma^X}\) is algebraically closed.
        \end{lemma}
        \begin{proof}
            By Lemma \ref{StanStructsHavePatching}, we have that \(\stan{\Gamma^X}\) has patching. Then, \(\Gamma^X\) is divisible as \(\Gamma\) is, and it is clear that \(\cal{P}(X)\) is Boolean. Thus, by the Algebraic Characterisation of Algebraically Closed Densely Valued \(\ell\)-Groups, \(\stan{\Gamma^X}\) is algebraically closed.
        \end{proof}

        Using the first of these Lemmas, we can produce a surprising fact about the relation between algebraically closed densely valued \(\ell\)-groups, and existentially closed \(\ell\)-groups.

        \begin{corollary}
            Let \(\cal{V} = (\cal{G},\cal{L},P)\) be an algebraically closed densely valued \(\ell\)-group. Then, \(\cal{G}\) isn't existentially closed.
        \end{corollary}
        \begin{proof}
            By Lemma \ref{ACValGrpProps} \((4)\) and \((5)\), \(\cal{G}\) has a weak order unit \(\nu\). However, by \cite[Section 2.1]{ConstructionOfExtClosedLGrpsViaFraisseLimits}, if \(\cal{G}\) is existentially closed, then there exists some \(a \in \cal{G}^+\) with \(a \neq 0\) and \(a \meet \nu = 0\). \Lightning. Thus, \(\cal{G}\) isn't existentially closed.
        \end{proof}

        This is of greater significance in view of the next section, as it tells us that every existentially closed densely valued \(\ell\)-group isn't existentially closed in its group sort. 

    \subsection{Existentially Closed Densely Valued \(\ell\)-Groups}

        In the previous section, we used the fact the lattice of a densely valued \(\ell\)-group was algebraically closed to show that the whole densely valued \(\ell\)-group was algebraically closed. We would like to mirror this technique when considering existentially closed densely valued \(\ell\)-groups.
        
        However, by \cite[Theorem \(8\)]{AlgClosedAndExtClosedDistroLats}, these are the atomless Boolean algebras. Clearly, \(\stan{\Gamma^X}\) has an atomic Boolean algebra as its lattice sort, and so we can't do this directly as we did above. Instead, we need to make use of the following lemma.

        \begin{lemma} \label{StanValExtensions}
            Let \(\Gamma\) be a divisible \(o\)-group, and \(X\) a non-empty set. Then, there is a valued \(\ell\)-group embedding \((I,J) : \stan{\Gamma^X} \to \stan{\Gamma^{X \times \{0,1\}}}\) given for \(f \in \Gamma^X\) by:
            \[I(f)((x,n)) = f(x) \text{ and } J(\{f \geq 0\}) = \{f \geq 0\} \times \{0,1\}.\]
        \end{lemma}
        \begin{proof}
            Immediate, as this is just the obvious extension of the diagonal embedding of \(\Gamma^X\) into \(\Gamma^{X \times \{0,1\}}\). 
        \end{proof}

        Now, we are in a position to characterise the existentially closed densely valued \(\ell\)-groups.

        \begin{theorem} [The Algebraic Characterisation of Existentially Closed Densely Valued \(\ell\)-Groups\label{ECValGrpAltDef}]
            Let \(\cal{V} = (\cal{G},\cal{L},P)\) be an algebraically closed densely valued \(\ell\)-group. Then, \(\cal{V}\) is existentially closed if and only if \(\cal{L}\) is \textbf{atomless} - for all \(l \in \cal{L}\) with \(\bot < l\), there exists some \(k \in \cal{L}\) with \(\bot < k < l\).
        \end{theorem}
        \begin{proof}
            \((\implies)\) First, we notice that by Lemma \ref{ACValGrpProps}, \(\cal{G} \not\cong \{0\}\). Now, by the Functional Representation Theorem, let \(\Gamma\) be a divisible \(o\)-group such that \(\cal{V} \subgrp \stan{\Gamma^{\spectrum{\cal{L}}}}\). Let \(l \in \cal{L}\) with \(l > \bot\). We have two cases:
            \begin{itemize}
                \item If \(l\) isn't an atom in \(\cal{P}(\spectrum{\cal{L}})\), then we have that:
                \[\stan{\Gamma^{\spectrum{\cal{L}}}} \models (\exists y : \bb{L})(\bot < y < l) \implies \cal{V} \models (\exists y : \bb{L})(\bot < y < l).\]

                Thus, \(l\) isn't an atom.

                \item Else, if \(l\) is an atom in \(\cal{P}(\spectrum{\cal{L}})\), then in particular \(l = \{\cal{F}\}\) for some \(\cal{F} \in \spectrum{\cal{L}}\). Now, let:
                \[(I,J) : \stan{\Gamma^{\spectrum{\cal{L}}}} \to \stan{\Gamma^{\spectrum{\cal{L}} \times \{0,1\}}})\]

                be the valued \(\ell\)-group embedding in Lemma \ref{StanValExtensions}. We see then that \(J(a) = \{\cal{F}\} \times \{0,1\} \supset \{(\cal{F},0)\}\). Thus, we have that:
                \[\stan{\Gamma^{\spectrum{\cal{L}} \times \{0,1\}}} \models (\exists y : \bb{L})(\bot < y < J(l)) \implies \cal{V} \models (\exists y : \bb{L})(\bot < y < l)\]

                and so \(l\) isn't an atom.
            \end{itemize}

            \((\impliedby)\) Let \(\phi(\bar{v},\bar{w})\) be an existential \(\lvlgrp\)-formula of sort \(\bb{G}^n \times \bb{L}^m\), \(\bar{a} \in \cal{G}^n\) and \(\bar{l} \in \cal{L}^m\). Let \(\cal{W} = (\cal{H},\cal{K},Q)\) be a densely valued \(\ell\)-group with \(\cal{V} \subgrp \cal{W}\) and \(\cal{W} \models \phi(\bar{a},\bar{l})\). By the Functional Representation Theorem, there is some divisible \(o\)-group \(\Gamma\) such that \(\cal{W} \subgrp \stan{\Gamma^{\spectrum{\cal{K}}}} \eqqcolon \hat{\cal{W}}\). In particular, \(\hat{\cal{W}}\) is algebraically closed by Lemma \ref{StanStructsAreAC}, and \(\hat{\cal{W}} \models \phi(\bar{a},\bar{l})\). Hence, by the Shen-Weispfenning theorem, let:
            \begin{itemize}
                \item \(\chi(\bar{p},\bar{w})\) be an \(\lvlgrp^-\)-formula of sort \(\bb{L}^{k + m}\); and

                \item For each \(i \in \{1,\ldots,k\}\), \(t_i(\bar{v}) : \bb{G}^n \to \bb{G}\) an \(\lvlgrp^-\)-term
            \end{itemize}

            such that the \(\lvlgrp^-\)-formula:
            \[\phi_{\norm{reduct}}(\bar{v},\bar{w}) \equiv (\exists y_1,\ldots,y_k : \bb{L})(\chi(\bar{y},\bar{w}) \land \bigland_{i = 1}^k(y_i = P(t_i(\bar{v}))))\]

            is existential, and equivalent to \(\phi(\bar{v},\bar{w})\) modulo \(\tdvlgrp^+\). In particular, we have that:
            \[(\exists y_1,\ldots,y_k \in \cal{P}(\spectrum{\cal{L}}))(\hat{\cal{W}} \models \chi(\bar{y},\bar{l}) \text{ and } (\forall i \in \{1,\ldots,k\})(y_i = \hat{Q}(t_i^{\hat{\cal{W}}}(\bar{a})))).\]
    
            However, we note that for \(1 \leq i \leq k\), \(y_i = \hat{Q}(t_i^{\hat{\cal{W}}}(\bar{a})) = P(t_i^{\cal{V}}(\bar{a}))\). Hence, as \(\cal{L}\) is atomless Boolean, and \(\chi(\bar{p},\bar{w})\) is an existential \(\lvlgrp\)-formula of sort \(\bb{L}^{k + m}\), then it is immediate from \cite[Theorem \(8\)]{AlgClosedAndExtClosedDistroLats} that \(\cal{V} \models \chi(y_1,\ldots,y_k,\bar{l})\). Hence, we have that:
            \[\cal{V} \models \phi_{\norm{reduct}}(\phi(\bar{a},\bar{l}) \implies \cal{V} \models \phi(\bar{a},\bar{l}).\]
    
            and so \(\cal{V}\) is existentially closed.
        \end{proof}

        As before, we can clearly see that this is a first-order characterisation - being atomless is defined by the \(\lvlgrp\)-sentence \((\forall x : \bb{L})(\bot < x \rightarrow (\exists y : \bb{L})(\bot < y < x))\). Hence, we have there is an \(\lvlgrp\)-theory of existentially closed densely valued \(\ell\)-groups, which we will denote by \(\tvlgrpec\). In particular, this is the model companion for \(\tvlgrp\). In the next section, we shall see that \(\tvlgrpec\) enjoys a very well-behaved model theory.

%% file: Chapters/model-theory-of-ec-valued-l-groups.tex
    \subsection{Completeness}

        The completeness of the theory of existentially closed densely valued \(\ell\)-groups is one of the easier properties to demonstrate, and also illustrates the general method we will use in the rest of this paper. This method is a sort of Ax-Kochen-Ershov transfer principle - we will transfer model-theoretic properties of atomless Boolean algebras to those of existentially closed densely valued \(\ell\)-groups, by way of the Shen-Weispfenning theorem.

        \begin{theorem} \label{ECDVAlGrpIsComp}
            \(\tvlgrpec\) is complete.
        \end{theorem}
        \begin{proof}
            Let \(\cal{V} = (\cal{G},\cal{L},P)\) and \(\cal{W} = (\cal{H},\cal{K},Q)\) be existentially closed densely valued \(\ell\)-groups, and \(\phi\) an \(\lvlgrp\)-sentence. By the Shen-Weispfenning theorem, there is an \(\lbdlat\)-sentence \(\chi\) such that \(\tvlgrpec \entails \phi \leftrightarrow \chi\). Then, as the theory of atomless Boolean algebras is complete by \cite[Section \(3.2.4\)]{Hodges}, we have that:
            \[\cal{V} \models \phi \iff \cal{L} \models \chi \iff \cal{K} \models \chi \iff \cal{W} \models \phi.\]

            Thus, \(\cal{V} \equiv \cal{W}\), and so \(\tvlgrpec\) is complete.
        \end{proof}

        When a proof of a theory's completeness is given, it is the author's view that this ought to be accompanied by an explicit example. Hence, we will below give a general method for constructing existentially closed densely valued \(\ell\)-groups, and as a corollary give an explicit countable model.

        \begin{lemma} \label{ECValGrpsViaDirectLimits}
            Let \(\cal{W} = (\cal{H},\cal{K},Q)\) be a densely valued \(\ell\)-group. Then, \(\cal{W}\) is existentially closed if and only if there exists a directed system \(\langle \cal{V}_i,(\alpha_{i,j},\beta_{i,j})\rangle_{i,j \in I}\) of algebraically closed densely valued \(\ell\)-groups such that:
            \begin{itemize}
                \item[\(i.\)] For each \(i,j \in I\), \((\alpha_{i,j},\beta_{i,j})\) is a valued \(\ell\)-group embedding;
                
                \item[\(ii.\)] For each \(i \in I\) and \(l \in \cal{L}_i\) with \(l \neq \bot\), there exists some \(j \in I\) with \(i \leq j\), and \(k \in \cal{L}_j\) such that \(\bot < k < \beta_{i,j}(l)\); and

                \item[\(iii.\)] \(\cal{W} = \underset{\rightarrow}{\lim} \cal{V}_i\).
            \end{itemize}
        \end{lemma}
        \begin{proof}
            \((\implies)\) Immediate, as \(\cal{V} = \underset{\rightarrow}{\lim} \cal{V}\), and this system clearly satisfies \(i.\). \\
            
            \((\impliedby)\) First, we note that \(\tvlgrpac\) is an \(\forall\exists\) \(\lvlgrp\)-theory. Hence, \(\cal{W} \coloneqq \underset{\rightarrow}{\lim} \cal{V}_i\) exists, and \(\cal{W} \models \tvlgrpac\) by \cite[Theorem \(2.4.6\)]{Hodges}. By the definition of a direct limit, let \((\Phi_i,\Psi_i) : \cal{V}_i \to \cal{W}\) be valued \(\ell\)-group embeddings such that this diagram commutes for all \(i,j \in I\) with \(i \leq j\):
            \begin{center}
                \begin{tikzcd}
                    & \cal{W} &                                                       \\
                    \cal{V}_i \arrow[rr, "{(\alpha_{i,j},\beta_{i,j})}" description] \arrow[ru, "{(\Phi_i,\Psi_i)}" description] &         & \cal{V}_j \arrow[lu, "{(\Phi_j,\Psi_j)}" description]
                \end{tikzcd}
            \end{center}
            
            Let \(\cal{W} = (\cal{H},\cal{K},Q)\), and let \(k \in \cal{K}\) with \(l \neq \bot\). By construction, there is some \(i \in I\) and \(l \in \cal{L}_i\) with \(k = \Psi_i(l)\) and \(l \neq \bot\). Thus, by \(i.\), there exists some \(j \in I\) with \(i \leq j\) and \(p \in \cal{L}_j\) such that \(\bot < p < \beta_{i,j}(l)\). It follows that:
            \[\bot = \Psi_j(\bot) < \Psi_j(p) < \Psi_j(\beta_{i,j}(l)) = \Psi_i(l) = k.\]

            Thus, \(\cal{K}\) is atomless, and hence \(\cal{W} \models \tvlgrpec\), as required.
        \end{proof}

        \begin{example} [\(2^n\)-Periodic Functions from \(\bb{N}\) to \(\bb{Q}\)\label{2nPeriodFuncEx}]
            For each \(n \in \bb{N}\), let \(I_n = \{1,2,3,\ldots,2^n\}\), and \(\cal{V}_n = \stan{\bb{Q}^{I_n}}\). Further, we define maps \(\alpha_n : \bb{Q}^{I_n} \to \bb{Q}^{I_{n + 1}}\) and \(\beta_n : \cal{P}(I_n) \to \cal{P}(I_{n + 1})\) by:
            \begin{align*}
                \alpha_n(f)(k) = & \begin{cases}
                    f(k) & 1 \leq k \leq 2^n \\
                    f(k - 2^n) & 2^n + 1 \leq k \leq 2^{n + 1}
                \end{cases} \\
                \beta_n(\{f \geq 0\}) = & \{\alpha_n(f) \geq 0\}.
            \end{align*}

            It is clear that \((\alpha_n,\beta_n)\) is a valued \(\ell\)-group embedding. Further, we note that for any \(n \in \bb{N}\) and \(\{f \geq 0\} \in \cal{P}(I_n)\) with \(\emptyset \subset \{f \geq 0\}\), we can define a map:
            \[g : I_{n + 1} \to \bb{Q} : k \mapsto \begin{cases}
                f(k) & 1 \leq k \leq 2^n \\
                -1 & 2^n + 1 \leq k \leq 2^{n + 1}
            \end{cases}\]

            and so we have that \(\emptyset \subset \{g \geq 0\} \subset \beta_n(\{f \geq 0\})\). Finally, we observe that \(\stan{\bb{Q}^{I_n}}\) is algebraically closed for each \(n \in \bb{N}\).  Hence, by Lemma \ref{ECValGrpsViaDirectLimits}, \(\cal{W} = \bigcup_{n \in \bb{N}} \cal{V}_n\) is an existentially closed densely valued \(\ell\)-group. \\

            We now observe that for any \(f : I_n \to \bb{Q}\), \(\alpha_n(f)\) is the map from \(I_{n + 1}\) to \(\bb{Q}\) which repeats \(f\) exactly twice. Hence, it is not too hard to see that \(\cal{W}\) is the set of functions from \(\bb{N}\) to \(\bb{Q}\) which have a period of \(2^n\) for some \(n \in \bb{N}\).
        \end{example}

        This example also gives us a better handle on the map \(P\), as the next lemma shows.

        \begin{lemma} \label{PAltDefInECValGrps}
            Let \(\cal{V} = (\cal{G},\cal{L},P)\) be an existentially closed valued \(\ell\)-group. Then, for all \(a,b \in \cal{G}^+\), \(P(-a) = P(-b)\) if and only if \(\{c \in \cal{G} \ | \ a \meet c = 0\} = \{c \in \cal{G} \ | \ b \meet c = 0\}\).
        \end{lemma} 
        \begin{proof}
            We remark our claim is equivalent to the \(\lvlgrp\)-sentence:
            \[(\forall a,b : \bb{G})((0 \leq a,b) \rightarrow ((P(-a) = P(-b)) \leftrightarrow (\forall c : \bb{G})(a \meet c = 0 \leftrightarrow b \meet c = 0))).\] 
            
            Hence, as \(\tvlgrpec\) is complete, it suffices to check it for the valued \(\ell\)-group of \(2^n\)-periodic functions from \(\bb{N}\) to \(\bb{Q}\), which we denote by \(\cal{W} = (\cal{H},\cal{K},Q)\).  Therefore, let \(f,g \in \cal{H}^+\), with periods \(2^n\) and \(2^m\) respectively. We check each direction separately:
            \begin{itemize}
                \item[\((\implies)\)] If \(\{-f \geq 0\} = \{-g \geq 0\}\), then in particular \(\{f = 0\} = \{g = 0\}\). Thus, let \(h \in \cal{G}\) with \(f \meet h = 0\). In particular, for all \(n \in \bb{N}\) with \(f(n) \neq 0\), \(h(n) = 0\). However, this clearly gives that \(h \meet g = 0\), and so by symmetry we are done. 

                \item[\((\impliedby)\)] Let \(c(n) = 1\) for all \(n \in \bb{N}\). Clearly, \(c \in \cal{H}\), and so we see that:
                \[\{f = 0\} = \{f \meet c = 0\} = \{g \meet c = 0\} = \{g = 0\}.\]

                As \(\{f = 0\} = \{-f \geq 0\}\), then we are done. \qedhere
            \end{itemize}
        \end{proof}

    \subsection{Quantifier Elimination}

        Another major model-theoretic property enjoyed by atomless Boolean algebras is that of quantifier elimination. However, in our setup, there is some subtlety. Our language omits a symbol for Boolean negation in the lattice sort, and so we will need to expand our language by a new function symbol. However, the Shen-Weispfenning theorem doesn't guarantee that we can eliminate group quantifiers in this new language. Thus, we need the following lemma first.

        \begin{lemma} \label{BoolAlgRemovalOfComp}
            Let \(\lbalg = \lbdlat \cup \{\lnot\}\), with \(\lnot\) a unary function symbol, and let \(T\) be the \(\lbalg\)-theory of Boolean algebras. Then, for all atomic \(\lbalg\)-formulae \(\phi(\bar{v})\), there exists an \(\lbdlat\)-formula \(\psi(\bar{v})\) such that \(T \entails (\forall \bar{a})(\phi(\bar{a}) \leftrightarrow \psi(\bar{a}))\).
        \end{lemma}
        \begin{proof}
            For any occurrence of \(\lnot s(\bar{v})\) appearing in \(\phi(\bar{v})\), we define a new \(\lbalg\)-formula \(\phi'(\bar{v},w)\) replacing every occurrence of \(\lnot s(\bar{v})\) by w. Then, \(\phi'(\bar{v},w)\) has fewer occurrences of \(\lnot\) then \(\phi(\bar{v})\), and it is clear that:
            \[T \entails (\forall \bar{a})(\phi(\bar{a}) \leftrightarrow (\exists b)(\phi'(\bar{a},b) \land (b \join s(\bar{a}) = \top) \land (b \meet s(\bar{a}) = \bot)).\]

            As only finitely many instances of \(\lnot\) appear in \(\phi(\bar{v})\), then we are done.
        \end{proof}

        Now, we can go via this lemma to show quantifier elimination in our expanded language.

        \begin{theorem} \label{QEForECValGrpsWBoolComp}
            Let \(\lvlgrpskolem = \lvlgrp \cup \{\lnot\}\), with \(\lnot : \bb{L} \to \bb{L}\). Let \(\tvlgrpecskolem = \tvlgrpec \cup \{\delta\}\), where \(\delta\) is the \(\lvlgrpskolem\)-sentence \((\forall a : \bb{L})((a \joina \lnot a = \top) \land (a \meeta \lnot a = \bot))\). Then:
            \begin{itemize}
                \item[\((1)\)] For all existentially closed densely valued \(\ell\)-groups \(\cal{V} = (\cal{G},\cal{L},P)\), there is a unique expansion \(\cal{V}^+\) of \(\cal{V}\) to an \(\lvlgrpskolem\)-structure, with \(\cal{V}^+ \models \tvlgrpecskolem\). 

                \item[\((2)\)] \(\tvlgrpecskolem\) has quantifier elimination, and in particular for every \(\lvlgrpskolem\)-formula \(\phi(\bar{v},\bar{w})\) of sort \(\bb{G}^n \times \bb{L}^m\), there exists:
                \begin{itemize}
                    \item A quantifier-free \(\lvlgrpskolem\)-formula \(\chi(\bar{p},\bar{w})\) of sort \(\bb{L}^{k + m}\); and

                    \item For each \(i \in \{1,\ldots,k\}\), an \(\lvlgrpskolem\)-term \(t_i(\bar{v}) : \bb{G}^n \to \bb{G}\)
                \end{itemize}

                such that:
                \[\tvlgrpecskolem \entails (\forall \bar{a} : \bb{G})(\forall \bar{x} : \bb{L})(\phi(\bar{a},\bar{x}) \leftrightarrow \chi(P(t_1(\bar{a})),\ldots,P(t_k(\bar{a})),\bar{x})).\]
            \end{itemize}
        \end{theorem}
        \begin{proof}
            \((1)\) We remark that, if \(\cal{V} \models \tvlgrpec\), then \(\cal{L}\) is Boolean. Thus, for each \(l \in \cal{L}\), there is a unique \(k \in \cal{L}\) with \(l \join k = \top\) and \(l \meet k = \bot\). Hence, our only choice for \(\lnot^{\cal{V}^+}\) is to send each \(l\) to this unique \(k\). \\

            \((2)\) Let \(\phi(\bar{v},\bar{w})\) be an \(\lvlgrpskolem\)-formula of sort \(\bb{G}^n \times \bb{L}^m\). By Lemma \ref{BoolAlgRemovalOfComp}, we can replace any subformula containing \(\lnot\) by a one not containing it to obtain an \(\lvlgrp\)-formula \(\phi'(\bar{v},\bar{w})\) such that:
            \[\tvlgrpecskolem \entails (\forall \bar{a} : \bb{G})(\forall \bar{x} : \bb{L})(\phi(\bar{a},\bar{x}) \leftrightarrow \phi'(\bar{a},\bar{x})).\]

            Now, by the Shen-Weispfenning Theorem, there is:
            \begin{itemize}
                \item An \(\lvlgrp\)-formula \(\chi(\bar{p},\bar{w})\) of sort \(\bb{L}^{k + m}\); and

                \item For \(i \in \{1,\ldots,k\}\), \(\lvlgrp\)-terms \(t_i(\bar{v},\bar{w}) : \bb{G}^n \to \bb{G}\)
            \end{itemize}

            such that the \(\lvlgrp\)-formula:
            \[\phi'_{\norm{reduct}}(\bar{v},\bar{w}) \equiv (\exists y_1,\ldots,y_k : \bb{L})(\chi(\bar{y},\bar{w}) \land \bigland_{i = 1}^k (y_i = P(t_i(\bar{v}))))\]

            satisfies \(\tvlgrpec \entails (\forall \bar{a} : \bb{G})(\forall \bar{x} : \bb{L})(\phi(\bar{a},\bar{x}) \leftrightarrow \phi'_{\norm{reduct}}(\bar{a},\bar{x}))\). As \(\chi(\bar{p},\bar{w})\) is an \(\lbalg\)-formula, then by \cite[Exercise \(8.5.17\)]{Hodges}, there is a quantifier-free \(\lbalg\)-formula \(\chi'(\bar{p},\bar{w})\) such that \(\tvlgrpecskolem \entails (\forall \bar{x},\bar{y} : \bb{L})(\chi(\bar{y},\bar{x}) \leftrightarrow \chi'(\bar{y},\bar{x}))\). It then follows that:
            \[\tvlgrpecskolem \entails (\forall \bar{a} : \bb{G})(\forall \bar{x} : \bb{L})(\phi(\bar{a},\bar{x}) \leftrightarrow \chi'(P(t_1(\bar{a})),\ldots,P(t_k(\bar{a})),\bar{x})).\]

            In particular, \(\tvlgrpecskolem\) has quantifier elimination, as required.
        \end{proof}

        We notice that the above theorem is somewhat stronger than just quantifier elimination - it tell us we can eliminate all quantifiers onto the lattice sort, in the style of Shen-Weispfenning. 

    \subsection{\(\omega\)-Categoricity and Automorphisms}

        The final model-theoretic property we will explore is that of \(\omega\)-categoricity. There is some hope a-priori that we can show this for \(\tvlgrpecskolem\), as atomless Boolean algebras are \(\omega\)-categorical. However, the addition of an \(\ell\)-group adds in a lot of complexity on the level of types, which the next lemma demonstrates:

        \begin{lemma} \label{2PeriodicFuncsOmitAType}
            For \(n \in \bb{N}\), let \(\phi_n(v)\) be the \(\lvlgrp\)-formula of sort \(\bb{G}^2\) given by:
            \[\phi_n(v,w) \equiv ((0 < v, w) \land (nv < w))\]

            and let \(\rho = \{\phi_n(v,w) \ | \ n \in \bb{N}, \ n \geq 1\}\). Then:
            \begin{itemize}
                \item[\((1)\)] There is some \(\hat{\rho} \in \typespace{\tvlgrpec}{(v,w)}\) such that \(\rho \subseteq \hat{\rho}\).

                \item[\((2)\)] Let \(\cal{V} = (\cal{G},\cal{L},P)\) be the valued \(\ell\)-group of \(2^n\)-periodic functions from \(\bb{N} \to \bb{Q}\)\footnote{See Example \ref{2nPeriodFuncEx}.}. Then, \(\cal{V}\) omits \(\hat{\rho}\).
            \end{itemize}
        \end{lemma}
        \begin{proof}
            \((1)\) Let \(S \subseteq \rho\) be a finite subset. We claim that:
            \[\tvlgrpec \entails (\exists a,b : \bb{G})\left(\bigland_{\phi(v,w) \in S} \phi(a,b)\right).\]

            Hence, let \(\cal{W} = (\cal{H},\cal{K},Q) \models \tvlgrpecskolem\), and \(N = \max \{n \in \bb{N} \ | \ \phi_n \in S\}\). Let \(b \in \cal{H}\) with \(0 < b\), and \(a = \frac{b}{N + 1}\). We see that:
            \[0 < a,b \text{ and } Na < b.\]

            Then, as for every \(k \leq N\), \(ka \leq Na\), we have that \((a,b)\) realise \(S\). In particular:
            \[\tvlgrpec \entails (\exists a,b : \bb{G})\left(\bigland_{\phi(v,w) \in S} \phi(a,b)\right)\]

            as required. Now, as every partial type can be extended to a complete type, we have that there is some \(\hat{\rho} \in \typespace{\tvlgrpec}{(v,w)}\) such that \(\rho \subseteq \hat{\rho}\), as required. \\

            \((2)\) Suppose \(f,g \in \cal{G}\) realise \(\hat{\rho}\). In particular, \(0 < f < g\), and so let \(n \in \bb{N}\) be such that \(0 < f(n) < g(n)\). Now, for \(k \in \bb{N}\), we have that \(kf(n) < g(n)\) (as \(kf < g\) for all \(k \in \bb{N}\)). Hence, as \(\bb{Q}\) is Archimedean, \(f(n) = 0\) \Lightning. Thus, \(\cal{V}\) omits \(\hat{\rho}\), as required.
        \end{proof}

        \begin{corollary}
            \(\tvlgrpec\) isn't \(\omega\)-categorical.
        \end{corollary}
        \begin{proof}
            By Lemma \ref{2PeriodicFuncsOmitAType}, there is a countable model of \(\tvlgrpec\) which omits a type. However, by the Ryll-Nardzewski theorem \cite[Thm \(7.3.1\)]{Hodges}, this cannot happen if \(\tvlgrpec\) is \(\omega\)-categorical.
        \end{proof}

        This suggests that the automorphism group of a countable model of \(\tvlgrpecskolem\) will be a difficult object to handle. Fortunately, the following result shows we can reduce automorphisms of an existentially closed densely valued \(\ell\)-group to just automorphisms of the underlying \(\ell\)-group.

        \begin{lemma}
            Let \(\cal{V} = (\cal{G},\cal{L},P)\) be an existentially closed densely valued \(\ell\)-group.Then:
            \begin{itemize}
                \item[\((1)\)] For each \(\phi : \cal{G} \to \cal{G}\) an \(\llgrp\)-automorphism, we have an \(\lvlgrp\)-automorphism \((\phi,\psi_{\phi}) : \cal{V} \to \cal{V}\) given by:
                \[\psi_{\phi} : \cal{L} \to \cal{L} : x \mapsto \begin{cases}
                    P(\phi(-a)) & a \in \cal{G}^+ \text{ such that } P(-a) = x\\
                    \bot & x = \bot
                \end{cases}\]

                \item[\((2)\)] The map \(\alpha : \norm{Aut}_{\llgrp}(\cal{G}) \to \norm{Aut}_{\lvlgrp}(\cal{V})\) given by \(\alpha(\phi) = (\phi,\psi_{\phi})\) is a group isomorphism.
            \end{itemize}
        \end{lemma}
        \begin{proof}
            \((1)\) First, we check that \(\psi_{\phi}\) is well-defined. Hence, let \(a,b \in \cal{G}^+\), and notice that by Lemma \ref{PAltDefInECValGrps}:
            \begin{align*}
                P(-a) = P(-b) \iff & (\forall c \in \cal{G})(a \meet c = 0 \iff b \meet c = 0) \\
                \iff & (\forall c \in \cal{H})(\phi(a) \meet c = 0 \iff \phi(b) \meet c = 0) \\
                \iff & P(\phi(-a)) = P(\phi(-b)) \\
                \iff & \psi_{\phi}(P(-a)) = \psi_{\phi}(P(-b)).
            \end{align*}

            Hence, \(\psi_{\phi}\) is well-defined and injective. Next, we show that \(\psi_{\phi}\) is a surjective bounded lattice morphism. Thus, consider:
            \begin{itemize}
                \item For \(k \in \cal{L}\), we have that there is some \(a \in \cal{G}^+\) such that \(P(-a) = k\). As \(\phi\) is an automorphism, let \(b \in \cal{G}^-\) such that \(\phi(a) = b\). Then, we see that \(k = P(-a) = P(\phi(-b)) = \psi_{\phi}(P(-b))\), and so \(\psi_{\phi}\) is surjective.

                \item Let \(a,b \in \cal{G}^+\), and consider that:
                \begin{align*}
                    \psi_{\phi}(P(-a) \joina P(-b)) = & \psi_{\phi}(P((-a) \join (-b))) \\
                    = & P(\phi((-a) \join (-b))) \\
                    = & P(\phi(-a)) \joina P(\phi(-b)) \\
                    = & \psi_{\phi}(P(-a)) \joina \psi_{\phi}(P(-b))
                \end{align*}

                and similarly for \(\meeta\). Hence, \(\psi_{\phi}\) is a lattice morphism.
            \end{itemize}

            Finally, by construction, we have that for \(a \in \cal{G}^+\) \(P(\phi(-a)) = \psi_{\phi}(P(-a))\), and so \((\phi,\psi)\) is an \(\lvlgrp\)-automorphism. \\

            \((2)\) By \((1)\), \(\alpha\) is a well-defined map. Further, the map \(\beta : \norm{Aut}_{\lvlgrp}(\cal{V}) \to \norm{Aut}_{\llgrp}(\cal{G})\) given by \(\beta((\phi,\psi)) = \phi\) is clearly a group morphism. Thus, it is sufficient to show that \(\alpha\) and \(\beta\) are inverses. Hence, consider:
            \begin{itemize}
                \item Let \(\phi \in \norm{Aut}_{\llgrp}(\cal{G})\). It is immediate that\((\beta \circ \alpha)(\phi) = \beta(\phi,\psi_{\phi}) = \phi\).

                \item Let \((\phi,\psi) \in \norm{Aut}_{\lvlgrpskolem}(\cal{V})\). We see that \((\alpha \circ \beta)(\phi,\psi) = \alpha(\phi) = (\phi,\psi_{\phi})\). Thus, we only need to show that \(\psi = \psi_{\phi}\). However, by definition, we see that for \(a \in \cal{G}^+\), \(\psi_{\phi}(P(-a)) = P(\phi(-a)) = \psi(P(-a))\), and so \(\psi_{\phi} = \psi\), as required.
            \end{itemize}
        \end{proof}